\documentclass[12pt]{amsart}

\usepackage[utf8]{inputenc}

\usepackage{hyperref}
\usepackage{graphicx}
\usepackage{amssymb}
\usepackage{float}
\usepackage{fullpage}
\usepackage{physics}
\usepackage{comment}
\usepackage{tikz}
\usetikzlibrary{positioning, arrows.meta, shapes.geometric}
\usepackage{mathtools}
\usepackage{bm}
\newtheorem{theorem}{Theorem}[section]
\newtheorem{lemma}[theorem]{Lemma}
\newtheorem{corollary}[theorem]{Corollary}
\newtheorem{proposition}[theorem]{Proposition}

\theoremstyle{definition}
\newtheorem{definition}{Definition}

\theoremstyle{remark}

\numberwithin{equation}{section}

\newcommand{\B}{\mathcal{B}}

\newcommand{\K}{\mathrm{K}}
\newcommand{\G}{\mathrm{G}}
\newcommand{\PP}{\mathcal{P}}
\newcommand{\C}{{\mathbb C}}
\newcommand{\R}{{\mathbb R}}

\newcommand{\Z}{{\mathbb Z}}
\newcommand{\1}{{\mathbf 1}}

\newcommand{\eps}{\varepsilon}

\newcommand{\SL}{\operatorname{SL}}

\renewcommand{\Re}{\operatorname{Re}}
\renewcommand{\Im}{\operatorname{Im}}
\renewcommand{\exp}{\operatorname{exp}}

\renewcommand{\H}{\mathcal{H}}
\renewcommand{\a}{\mathrm{a}}
\newcommand{\n}{\mathrm{n}}
\renewcommand{\k}{\mathrm{k}}
\newcommand{\g}{\mathrm{g}}
\newcommand{\h}{\mathrm{h}}

\renewcommand{\c}{\mathfrak c}
\renewcommand{\d}{\mathrm d}
\renewcommand{\n}{\mathrm{n}}

\newcommand{\sgn}{\operatorname{sgn}}
\renewcommand{\pmod}[1]{\,\,(\mathrm{mod}\,{#1})}

\title{Weighted averages of $\SL_2(\R)$ automorphic kernel \\ Part I: non-oscillatory functions}

\author{Lasse Grimmelt}
\address{Mathematical Institute, University of Oxford\\ Radcliffe Observatory Quarter, Woodstock Rd\\
Oxford OX2 6GG\\
UK}
\email{lasse.grimmelt@maths.ox.ac.uk}
\author{Jori Merikoski}
\address{Mathematical Institute, University of Oxford\\ Radcliffe Observatory Quarter, Woodstock Rd\\
Oxford OX2 6GG\\
UK}
\email{jori.merikoski@maths.ox.ac.uk}
\date{}
\subjclass[2020]{11F72, 11N75 primary}

\begin{document}
\begin{abstract}
We prove a theorem that evaluates weighted averages of sums parametrised by congruence subgroups of $\SL_2(\Z)$. In the proof, spectral methods are applied directly to the automorphic kernel instead of going over sums of Kloosterman sums. In number theoretical applications this better preserves the specific symmetries throughout the application of spectral methods.  In a separate paper we apply the main theorem to quadratic polynomials and obtain new results about their greatest prime factor and the equidistribution of their roots to prime moduli. 
\end{abstract}

\maketitle

\tableofcontents
\section{Introduction}

The determinant equation $ad-bc=h$ is ubiquitous in number theory and often we wish to count its integer solutions in dyadic ranges and subject to congruence restrictions. The go-to approach has been to apply Poisson summation and show cancellation along the resulting  \emph{sums of Kloosterman sums}. The \emph{Kuznetsov trace formula} is the crucial ingredient in this approach. The seminal work of Deshouillers and Iwaniec \cite{deshoullieriwaniec} has since produced countless applications (cf. for instance, \cite{BFI,motohashi,petrowyoung,young}), colloquially referred to as the \emph{Kloostermania}. 

In \cite{GM}, the authors developed an alternative approach that directly applies spectral theory to a Poincaré type series. In this paper, we advance this perspective further and interpret the counting problem as a weighted average of a two-variable $\SL_2(\R)$ automorphic kernel function. In contrast to \cite{GM} that still relied on one application of the spectral large sieve, this allows us to sidestep the spectral large sieve and with it Kuznetsov's formula and Kloosterman sums completely. This paper is the part I of a two-part series, in the part II we will generalise to oscillatory weight functions.

Our main theorem inherits all benefits of the results in \cite{GM}, for example, it is easier to apply compared to the classical Kloosterman sum approach in cases which would require delicate considerations about cusps and scaling matrices. Furthermore, it has a new advantage when averaging over the level variable is available, as it fully preserves existing divisor switching symmetry. See Section \ref{sec:applications} for announcement of applications to quadratic polynomials. The proofs of these applications are based on Theorem \ref{thm:kinvtechnical} and appear in the companion paper \cite{GMquadratic}. In Section \ref{sec:examples} we give further simple motivational examples.

We now introduce some notation for the statement of main theorem. Let $\Gamma \subseteq \SL_2(\Z)$ be a congruence subgroup of level $q \geq 1$ and  fix a group character $\chi( \smqty(a & b \\ c& d)) = \chi(d)$ for $\chi$ a Dirichlet character modulo $q.$ We let $\Gamma$  act on the Lie group $\G := \SL_2(\R)$ by multiplication from the left $\g \mapsto \gamma \g$.  Given a compactly supported function $F:\G \to \C$, we construct the \emph{automorphic kernel function} $\mathcal{K}_{\Gamma,\chi}F:\G \times \G \to \C$ by
\begin{align*}
      ( \mathcal{K}_{\Gamma,\chi} F)(\tau_1,\tau_2)&:=\sum_{\gamma \in \Gamma} \overline{\chi}(\gamma) F(\tau_1^{-1} \gamma \tau_2).
\end{align*}
We shall write $\mathcal{K}=\mathcal{K}_{\Gamma,\chi}$ when  $\Gamma$ and $\chi$ are clear from the context.

If $F$ is a smooth bump function and its support is large  and not too skewed compared to the level $q$, we expect that the sum over $\Gamma$ is well approximated by the corresponding integral over $\G$ against the Haar measure $\d \g$, normalised by the volume of the fundamental domain $\Gamma \backslash \G$. This leads us to define the discrepancy $\Delta=\Delta_{\Gamma,\chi}$
\begin{align*}
  \Delta F (\tau_1,\tau_2)  &:=    \mathcal{K}F(\tau_1,\tau_2) - \frac{\mathbf{1}_{\chi \, \mathrm{principal}}}{|\Gamma \backslash \G|} \int_\G F( \g ) \d \g.
\end{align*}
The main goal is to show that this discrepancy is small, at least in some average sense. 

\begin{definition}[Weighted average discrepancy]
   For a compactly supported $F:\G \to \C$ and finitely supported weights $\alpha_1,\alpha_2: \G \to \C$ we define the \emph{weighted average discrepancy}
   \begin{align*}
       \langle \alpha_1 | \Delta F | \alpha_2 \rangle = \sum_{\tau_1, \tau_2} \overline{\alpha_1(\tau_1)} \alpha_2(\tau_2) \Delta F(\tau_1,\tau_2). 
   \end{align*}
\end{definition}
We also require a definition that encapsulates the notion of a smooth dyadic bump function.
\begin{definition}[Smooth dyadically supported function]\label{def:SandS}
  For $n,J \in \Z_{>0}$, $\delta > 0,$ and $X_1,\dots,X_n \in \R_{>0}$  define the space of $J$ times differentiable dyadically supported functions. 
\begin{align*}
    C_\delta^J (X_1,\dots,X_n) :=\big\{f\in& C^J(\R^n): \, f(x_1,\dots,x_n) \,\, \text{supported on} \,  \\
    &(|x_1|,\dots,|x_n|) \in [X_1,2X_1] \times \cdots \times [X_n,2X_n], \\
    &  \|\partial_{x_1}^{J_1} \cdots \partial_{x_n}^{J_n} f\|_\infty \leq  \prod_{i \leq n} (\delta  X_i)^{-J_i} \quad \text{for all} \quad 0 \leq J_1+\cdots+J_n  \leq J  \big\}.
\end{align*}
\end{definition}

 We define  for any $R > 0$  and  $\g = \smqty(a &b \\ c& d) \in \G$ the $R$-skewed  hyperbolic size (cf. \eqref{eq:ulowerbound} also)
\begin{align*}
    u_R(\g)  := \tfrac{1}{4} (a^2+(b/R)^2+(cR)^2+d^2-2).
\end{align*}
For any statement $S$ we let $\mathbf{1}\{S\}$ denote the indicator function that $S$ holds. We denote  $\theta=\theta(\Gamma,\chi) := \max\{0,\Re(\sqrt{1/4-\lambda_1(\Gamma,\chi)})\}$ with $\lambda_1(\Gamma,\chi)$ denoting the smallest positive eigenvalue of the hyperbolic Laplacian for $\Gamma$ with character $\chi$. The \emph{Selberg eigenvalue conjecture} states that
 we have $\theta=0$ and the best unconditional bound for congruence subgroups is $\theta \leq 7/64$ by Kim and Sarnak \cite{KimSarnak}.  

 The most technical version of our main result is given in Theorem \ref{thm:technical}, for now we state the following special case.
\begin{theorem} \label{thm:mainblackbox}
Let $\Gamma \subseteq \SL_2(\Z)$ be a congruence subgroup with a group character $\chi$. Let $\alpha_1,\alpha_2 \in \mathcal{L}_c(\G)$. Let $A,C,D,\delta > 0$  with $AD > \delta$ and let $f \in C^{10}_\delta (A,C,D)$. Denote $F:\G \to \C, \, F(\smqty(a & b \\ c& d)):= f(a,c,d)$ and denote $R_1=A/C, R_2=D/C$. Then for any choice $X_0X_1X_2 \geq AD $ with $X_0,X_1,X_2 \geq 1$ we have
\begin{align*}
 \langle \alpha_1 |\Delta F| \alpha_2\rangle  \ll \delta^{-O(1)} (AD)^{1/2+o(1)}X_0^{\theta} \sqrt{ \, \langle  \alpha_1| \Delta k_{X_1^2,R_1}| \alpha_1\rangle \, \langle  \alpha_2| \Delta k_{X_2^2,R_2}| \alpha_2 \rangle},
\end{align*}
where $\langle \alpha_i|\Delta k_{X_i^2,R_i}| \alpha_i\rangle \geq 0$ for certain smooth functions $k_{Y,R} :\G \to [0,1]$ which satisfy
\begin{align*}
    k_{Y,R}(\g)  \leq  \frac{\mathbf{1}\{u_R(\g) \leq Y\}}{\sqrt{1+u_R(\g)}}.
\end{align*}
\end{theorem}
Theorem \ref{thm:mainblackbox} starts with an automorphic kernel and, after passing through a spectral expansion and Cauchy-Schwarz, requires the estimation of two new kernel functions on the right-hand side. On the one hand, corresponding to the \emph{diagonal contribution} of the Cauchy-Schwarz, the kernel $k_{X_i^2,R_i}$ have concentrated mass $\asymp 1$ in the skewed ball $|a|,|d| \leq 1, |b| \leq R_i, |c| \leq 1/R_i$. Outside this ball, the kernel decay due to $1/\sqrt{1+u_{R_i}(\g)}$. On the other hand, corresponding to the \emph{off-diagonal} contribution, the limiting case in the support occurs when $|a|,|d| \asymp X_i$, $|b| \asymp R_i X_i,$ $|c| \asymp R_i^{-1} X_i$. There the new kernel $k_{X_i^2,R_i}$ are supported around $(A_i,C_i,D_i)$ which satisfy $ A_i D_i = X_i^2, \frac{A_i}{C_i} = \frac{D_i}{C_i} = R_i $, analogously to the original function $F$. Therefore, if we consider $\alpha_1=\alpha_2$, $R_1=R_2,$ $X_0=1$, and $X_1=X_2= \sqrt{AD}$, then the statement of Theorem 1.1 is trivial. In applications, we win over the trivial bound by taking $X_0 > (AD)^\eps$, which exploits the existence of a spectral gap $\theta < 1/2$.  

The statement has other symmetries, for instance, as a corollary we see that the it stays true if we replace $\Gamma$  by the conjugate $\a^{-1} \Gamma \a$ for any diagonal matrix $\a \in \G$.

\subsection{Applications to quadratic polynomials} \label{sec:applications}
In the companion paper \cite{GMquadratic}, we consider quadratic polynomials of the shape $an^2+h$ and obtain new results about their greatest prime factor, the equidistribution of their roots to prime moduli, and a certain divisor problem which is used in the recent work of the second author \cite{mx2y3}. 

With respect to the greatest prime factor, we obtain the following theorem that improves upon \cite{mlargepf,pascadi2024large} in two aspects. First, we unconditionally obtain the exponent 1.312 that previously was only available under Selberg's eigenvalue conjecture.  Second, we get uniformity in the shift $h$ under an assumption involving the Legendre symbol $(\tfrac{x}{p})$. To state the assumption, we define
\begin{align*}
  \varrho_{a,h}(k):= \#\{ \nu \in \Z/k\Z: a\nu^2+h \equiv 0 \pmod{k} \}.
\end{align*}
This is a multiplicative function with $ \varrho_{a,h}(p) = 1 + (\frac{-ah}{p})$ for $\gcd(a,p)=1$. In contrast to all previous works, the proof is independent of progress towards the Selberg eigenvalue conjecture -- any fixed spectral gap would give the same quality of result.
\begin{theorem}\label{thm:primefactor}
There exists some small $\eps >0$ such that the following holds for all $X > \eps^{-1}$. Let $1 \leq h\leq X^{1+\eps}$ be square-free and let $1 \leq a\leq X^{\eps}$ with $\gcd(a,h)=1$. Suppose that for any $X^{\eps}<  Y < Z \leq X^2$ we have
\begin{align*}
  \sum_{Y \leq p < Z} \frac{\log p}{p} \varrho_{a,h}(p)\leq (1+\eps) \log Z/Y + 1/\eps.
\end{align*}
 Then there exists $n\in [X,2X]$ such that the greatest prime factor of $a n^2+h$ is greater than $X^{1.312}$.  
\end{theorem}
For $ah \leq X^{\eps^2}$ the hypothesis can be shown unconditionally by the zero-free region, since a possible Siegel zeros would help to make the sum small. It seems likely that this hypothesis is a fundamental obstruction, as its contraposition means that the $an^2+h$ have more  small prime factors than expected. 
 
In similar spirit, we can also obtain uniformity in $h$ for the equidistribution of roots of quadratic congruences to prime moduli, generalising the work of Duke, Friedlander, and Iwaniec \cite{DFIprimes}. We again need to include the distribution of $\varrho_{a,h}(p) $ in the statement. However, here the issue is of opposite nature and the theorem becomes trivial if $\varrho_{a,h}(p)$ is $0$ too often, i.e. if there is a Siegel zero for the real character $(\frac{-ah}{\cdot})$. 
\begin{theorem}\label{thm:roots}
    Let $1 \leq h\leq X^{1+o(1)}$ be square-free and let $1 \leq a\leq X^{o(1)}$ with $\gcd(a,h)=1$. Then for any $0\leq \alpha <\beta \leq 1$ we have
    \begin{align*}
   \# \Bigl\{(p,\nu): p\leq X, \, a\nu^2+h\equiv 0 \pmod{p}, \, \frac{\nu}{p} \in ( \alpha, \beta] \Bigr\} =  (\beta-\alpha) \sum_{p \leq X} \varrho_{a,h}(p) + o\bigl(\pi(X)\bigr) .
    \end{align*}
\end{theorem}

Finally, in \cite{GMquadratic} we show a technical variant of the divisor problem for the binary polynomial $ax^2+by^3$. This was used by the second author \cite{mx2y3} to show the following result about primes of the shape $ax^2+by^3$, conditional to the hypothesis that certain Hecke eigenvalues exhibit square-root cancellation along the values of binary cubic forms.
\begin{theorem}\label{thm:x2y3}   Let $a,b > 0$ be coprime integers. Assume that \cite[Conjecture $\mathrm{C}_a(\eps)$]{mx2y3} holds for all $\eps >0$.  Then 
    \begin{align*}
        \sum_{x \leq X^{1/2}} \sum_{y \leq X^{1/3}} \Lambda(ax^2+by^3) = (1+o(1)) X^{5/6}.
    \end{align*}
\end{theorem}

\subsection{Examples} \label{sec:examples}
It is instructive to consider what Theorem \ref{thm:mainblackbox} says if  $\Gamma=\Gamma_0(q)$, $\chi=1$, $\delta \gg 1,$ and $\alpha_1=\alpha_2= \delta_I$, that is, when
\begin{align*}
     \langle \alpha_1 |\Delta F| \alpha_2\rangle  =  \sum_{\gamma \in \Gamma}  F(\gamma ) - \frac{1}{|\Gamma \backslash \G|} \int_\G F( \g ) \d \g.
\end{align*}
An application of Theorem \ref{thm:mainblackbox} shows that this discrepancy is bounded by
\begin{align*}
  \ll (AD)^{1/2} X_0^{\theta} K_1^{1/2} K_2^{1/2},
\end{align*}
where, by using $0 \leq \langle \alpha_i |\Delta k_{X_i^2,R_i } | \alpha_i \rangle \leq \langle \alpha_i | \mathcal{K} k_{X_i^2,R_i }  | \alpha_i \rangle $, we have
\begin{align*}
    K_i  \leq   \sum_{\gamma \in \Gamma_0(q)} k_{X_i,R_i}(\gamma) \leq \sum_{\smqty(a & b \\ c & d) \in \Gamma_0(q)} \frac{\1\{|a|^2+|b/R_i|^2+|R_ic|^2+|d|^2\leq X_i^2\}}{\sqrt{|a|^2+|b/R_i|^2+|R_ic|^2+|d|^2}}.
 \end{align*}
Terms with $b=0$ contribute $\ll 1+(R_iq)^{-1}$ and terms with $c=0$ contribute at most $\ll 1+R_i$. Terms with $bc\neq 0$ can be estimated with a divisor bound by
\begin{align*}
    \ll X_i^{o(1)} \sum_{\substack{ad \equiv 1 \pmod{q} \\ |a|,|d| \leq X_i }} \frac{1}{\sqrt{|a|^2 + |d|^2}} \ll X_i^{o(1)}  \Bigl(1 +\frac{X_i}{q}\Bigr).
\end{align*}
Choosing $X_1=X_2=q$ we get
\begin{align*}
    \bigg|\sum_{\gamma \in \Gamma}  F(\gamma ) - \frac{1}{|\Gamma \backslash \G|} \int_\G F( \g ) \d \g\bigg| \ll (AD)^{1/2+o(1)} \bigg( 1+\frac{AD}{q^2} \bigg)^\theta  \bigg( 1 +\frac{A}{C} + \frac{C}{A q} \bigg)^{1/2} \bigg(  1 +\frac{D}{C} + \frac{C}{D q} \bigg)^{1/2}.
\end{align*}
If the ranges are skewed, that is, if one of the ratios $A/C,C/(Aq), D/C, C/(Dq)$ is larger than $(AD)^{\eps}$, then the sum $\sum_{\gamma \in \Gamma}  F(\gamma )$ may be evaluated more easily by applying Poisson summation formula to one of the variables. If the ranges are not skewed, then the bound simplifies to $(AD)^{1/2+o(1)} ( 1+\frac{AD}{q^2} )^\theta,$ which is non-trivial provided that $q \leq (AD)^{1/2-\eps}$. For this simple example, the same bounds can be obtained via sums of Kloosterman sums and the spectral large sieve \cite{deshoullieriwaniec}. In general, this illustrates how our kernel-based approach leads to results of at least the same strength as using the spectral large sieve. 

As a second example of how to apply Theorem \ref{thm:mainblackbox}, we show the following version of \cite[Theorem 1.1]{GM} that for practical purposes has a slightly better dependency on  $\theta$. In particular, this indicates how the results in \cite{GM} are contained in  the approach here. 
\begin{corollary} 
Let $\mathcal{M} \subseteq \SL(\R)$,\, $q_1,q_2\in \Z_{>0}$ and denote $\Gamma = \Gamma_2(q_1,q_2)$. Assume that $\mathcal{M}$ is left 
$\Gamma$-invariant with finitely many orbits $\Gamma \backslash \mathcal{M}$ and let $\alpha:\mathcal{M} \to\C$ be left $\Gamma$-invariant.  Let $A,C,D,\delta > 0$  with $AD > \delta$ and let $f \in C^{10}_\delta (A,C,D)$. Denote $F:\G \to \C, \, F(\smqty(a & b \\ c& d)):= f(a,c,d)$. Then  we have
\begin{align*}
   \bigg|  \sum_{\g \in \mathcal{M}}  \alpha(\g) F(\g)  -  \frac{1}{|\Gamma \backslash \SL(\R) |} \sum_{\tau \in \Gamma \backslash \mathcal{M}} \alpha(\tau)  \int_{\G} F(\g) \, \d \g \bigg|  \ll (AD)^{o(1)} \delta^{-O(1)}   \Bigl(1+\frac{AD}{X_1}\Bigr)^{\theta} \sqrt{ AD \, \mathcal{K} \, \mathcal{R}},
\end{align*}
where $X_1 = \max\{q_1 q_2,  q_2 A /C, q_1 C  /A\}$,
\begin{align*}
\mathcal{K}& = \sum_{\substack{\g =\smqty(a & b \\c & d) \in \mathcal{M}^{-1} \mathcal{M} \\  |a|+|b|C/D+|c|D/C+|d| \,\leq 10}}\Bigl|\sum_{ \substack{\tau\in \Gamma \backslash \mathcal{M} 
 \\ \tau \g \in \mathcal{M}}} \alpha(\tau) \overline{\alpha(\tau g)} \Bigr| \quad \quad\text{and}  \quad \quad \mathcal{R} =    1 + \frac{A}{C q_1}+\frac{C}{A q_2}.
\end{align*}
\end{corollary}
\begin{proof}
  Since $ \Gamma_2(q_1,q_2) = \mathrm{a}[q_1]  \Gamma_0(q_1q_2) \mathrm{a}[1/q_1],$ by conjugation symmetry it suffices to prove the claim for $\Gamma_0(q)$. Denoting $T=\Gamma \backslash \mathcal{M}$,  $\alpha_1 = \delta_{I}$, $\alpha_2 = \alpha $, the left-hand side becomes $ |\langle\alpha_1|\Delta F| \alpha_2 \rangle | $. Thus, applying Theorem \ref{thm:mainblackbox} with  $X_2=1$ we have
\begin{align*}
| \langle \alpha_1 |\Delta F| \alpha_2\rangle | \ll \delta^{-O(1)} (AD)^{1/2+o(1)} \Bigl(\frac{AD}{X_1}\Bigr)^{\theta}  \sqrt{ \, \langle \alpha_1|  \Delta k_{X_1,R_1} | \alpha_1\rangle  \, \langle \alpha_2| \Delta k_{1,R_2} | \alpha_2\rangle}  ,
\end{align*}
The first kernel is bounded by the same argument as in the first example, making the choice $X_1=\max\{q,q R_1 , R_1^{-1}\}$. The second kernel satisfies by positivity and the triangle inequality
 \begin{align*}
   0 \leq  \langle \alpha_2| \Delta k_{1,R_2}| \alpha_2\rangle \leq  \langle\alpha_2|\mathcal{K}k_{1,R_2}| \alpha_2\rangle  \leq \mathcal{K}.
 \end{align*}
\end{proof}

\subsection{A generalisation to linear functionals} 
Theorem \ref{thm:mainblackbox} applies more generally to linear functionals in place of the finitely supported weights. This language also allows for a conceptual proof sketch, see the next subsection. Let $\mathcal{L}(\G)$  denote the set of linear functionals $C(\G) \to \C$. For $\alpha \in \mathcal{L}(\G)$ denote the image of $ f \in C(\G)$ by $\langle f\rangle_\alpha  := \alpha(f) \in \C$. 
\begin{definition}[Compactly supported linear functional]
We define the set of compactly supported linear functionals via
\begin{align*}
    \mathcal{L}_c(\G) := \bigcup_{K \subseteq \G \, \text{compact}} \bigcup_{c>0}\Bigl\{\alpha \in \mathcal{L}(\G):     |  \langle f \rangle_\alpha| \leq c \sup_{\g \in K}  |f(\g)| \text{ for all } f \in C(\G)\Bigr \}.
\end{align*}
\end{definition}
We make this definition to ensure convergence of the spectral expansion in the proof. It can be relaxed to include more general functionals, as long as the relevant sums are absolutely convergent. In practice, we are interested in functionals given by weighted sums $f \mapsto \sum_{\tau \in T}  \alpha(\tau)f(\tau)$ for some finite set $T \subseteq 
 \G$ and weight $\alpha:T \to \C$. For any $\alpha \in \mathcal{L}(\G)$ we define the complex conjugate by 
\begin{align*}
  \langle f \rangle_{\overline{\alpha}} =  \overline{ \langle \overline{f} \rangle_{\alpha} } .
\end{align*}
\begin{definition}[Sesquilinear form induced by a binary function]
For any  $F \in C(\G \times \G)$ we define the sesquilinear form  $\mathcal{L}_c(\G) \times \mathcal{L}_c(\G) \to \C$ 
\begin{align*}
 \langle \alpha_1 | F | \alpha_2 \rangle := \langle \langle F \rangle_{\overline{\alpha_1}} \rangle_{\alpha_2},
\end{align*}
where $\overline{\alpha_1}$ acts on the left variable of $F(\g_1,\g_2)$ and $\alpha_2$  acts on the right variable.
\end{definition}
Restriction to $\mathcal{L}_c(\G)$ means that this is well-defined, that is, $\langle F \rangle_{\overline{\alpha_1}}$ defines a continuous function in the second variable, and we have $\langle F \rangle_{\overline{\alpha_1}} \rangle_{\alpha_2} = \langle \langle F \rangle_{\alpha_2} \rangle_{\overline{\alpha_1}},$  where in both expressions $\overline{\alpha_1}$ acts on the left variable and $\alpha_2$  acts on the right variable. This may be viewed as a tensor product of $\overline{\alpha_1}$ and $ \alpha_2$, that is, for $F(\tau_1,\tau_2) =\overline{f_1(\tau_1)} f_2(\tau_2)$ we have
\begin{align*}
   \langle \alpha_1 | F | \alpha_2 \rangle =  \overline{\langle f_1\rangle_{\alpha_1}}   \langle f_2\rangle_{\alpha_2}
\end{align*}
Letting the binary function be $\Delta F (\tau_1,\tau_2)$, this generalizes the weighted average discrepancy to linear functionals and Theorem \ref{thm:mainblackbox} applies to the finitely supported weights replaced by $\alpha_1,\alpha_2 \in \mathcal{L}_c(\G)$.

\subsection{Sketch of the proof of Theorem \ref{thm:mainblackbox}}
As suggested by the notation, Theorem \ref{thm:mainblackbox} may be seen as a quantitative variant of the popular  $L^\infty L^2 L^2$ inequality
\begin{align} \label{eq:LinftyL2L2}
    |\langle \mathbf{a}_1 | \mathbf{X}_1^\ast \mathbf{X}_0 \mathbf{X}_2 | \mathbf{a}_2 \rangle| \leq \|\mathbf{X}_0\|_{\mathrm{op}} \sqrt{\langle \mathbf{a}_1 | \mathbf{X}_1^\ast \mathbf{X}_1 | \mathbf{a}_1 \rangle \, \langle \mathbf{a}_2 | \mathbf{X}_2^\ast \mathbf{X}_2 | \mathbf{a}_2 \rangle},
\end{align}
which is true for any $\mathbf{a}_1,\mathbf{a}_2 \in \mathbf{H} $  and bounded operators $\mathbf{X}_0,\mathbf{X}_1,\mathbf{X}_2$ on  a Hilbert space $\mathbf{H}$ with the inner product $\langle\cdot| \cdot \rangle$. We now explain this correspondence.

The first step in the proof of Theorem \ref{thm:mainblackbox} is to reduce to the balanced case $R_1=R_2=1$ via what we call \emph{unskewing}. Denote $\a[y] := \smqty(y^{1/2} & \\ & y^{-1/2})$ and
\begin{align*}
    \widetilde{F}(\g) := F(\a[R_1] \g \a[R_2]^{-1})
\end{align*}
and define linear functionals $\widetilde{\alpha}_j$ via
\begin{align*}
 \langle f \rangle_{\widetilde{\alpha}_j} =   \langle r_{\a[R_j]}f \rangle_{\alpha_j}. 
\end{align*}
Then we have
\begin{align*}
 \langle  \alpha_1|  \Delta F | \alpha_2 \rangle=       \langle  \widetilde{\alpha}_1|  \Delta \widetilde{F} | \widetilde{\alpha}_2\rangle \text{ and }  \langle  \alpha_j|  \Delta k_{X_j^2,R_j} | \alpha_j \rangle = \langle  \widetilde{\alpha}_j|  \Delta k_{X_j^2,1} | \widetilde{\alpha}_j \rangle. 
\end{align*}
Noting that $\widetilde{F}\smqty(a & b \\ c& d))= \widetilde{f}(a,c,d)$ with $\widetilde{f} \in C^{10}_\delta(\sqrt{AD},\sqrt{AD},\sqrt{AD})$, it then suffices to prove the theorem for $R_1=R_2=1$.

Let us denote the convolution on $\G$ by $(F_1\ast F_2) (\g) = \int_\G F_1(\h) F_2(\h^{-1} \g)\d \h $.  We sketch the remaining proof for the special case that $\chi=1$ and $F=\overline{F_1}\ast F_0\ast F_2$, where $F_j:\G \to \C$ are  bi-$\mathrm{K}$-invariant smooth bump functions supported on $u(\g) = u_1(\g) \asymp X_j$ with $AD=X_0X_1X_2$. While the general case requires a more concrete approach compared to this abstract sketch, the input is morally the same. The main difference is that the involved functions do not factorise neatly and are not bi-$\mathrm{K}$-invariant. The factorisation is still approximately true on the spectral side and, as a result of the unskewing, the latter can be accounted for by incorporating wiggle room in the left and right $\mathrm{K}$-types up to height at most $(AD)^{o(1)}$. 

For bi-$\K$-invariant $F$ the kernel function $\mathcal{K}F(\tau_1,\tau_2)$ is right-$\K$-invariant in both coordinates. The compactly supported linear functionals on $C( \G / \mathrm{K})$ form a vector space $\mathcal{L} = \mathcal{L}_c(\G/\K)$. 
For any small $\delta > 0$, let $\psi_{\delta}(u(\g)) =  \psi(\delta^{-1} u(\g)) $ for $\psi:[0,\infty) \to \R_{\geq 0}$ a smooth function supported on $[0,1]$  and $\psi=1$ on $[0,1/2]$.  We define the normalised convolution 
\begin{align*}
    h_{\delta} = \frac{1}{\|\psi_{\delta} \ast \psi_{\delta}\|_1} \psi_{\delta} \ast \psi_{\delta},
\end{align*}
so that $h_{\delta} $ approximates the Dirac mass on  $u(\g)=0$ with respect to the Haar measure $\d \g$. A key observation is that this defines an inner product on $\mathcal{L}$ via
\begin{align*}
    \langle \alpha_1,\alpha_2\rangle_{\delta} :=   \langle \alpha_1 | \Delta h_\delta | \alpha_2 \rangle.
\end{align*}
 The positive-definiteness $ \langle \alpha,\alpha\rangle_{\delta}  > 0$ for $\alpha \neq 0$ follows by the spectral expansion of an automorphic kernel (the pretrace formula, Proposition \ref{prop:kernelexpansion}) and the spectral positivity of convolutions of real valued even functions (cf. Lemma \ref{le:convosquare}). Conjugate symmetry follows by using $h_\delta(\g) =  h_\delta(\g^{-1})$ to get $\Delta h_\delta (\tau_1,\tau_2) = \Delta h_\delta (\tau_2,\tau_1)$.

A compactly supported $f \in C(\mathrm{K} \backslash \G / \mathrm{K})$ defines an operator $\mathbf{X}_f:\mathcal{L} \to \mathcal{L}$  via convolution
\begin{align*}
    \langle F \rangle_{\mathbf{X}_f\alpha}  =      \langle  F \ast f \rangle_{\alpha}.
\end{align*}
 Then, since Dirac delta is the unit for convolutions, for $F=\overline{F_1}\ast F_0 \ast F_2$ we have 
 as $\delta \to 0$
\begin{align} \label{eq:diraclimit}
     \langle \mathbf{X}_{F_1} \alpha_1, \mathbf{X}_{F_0} \mathbf{X}_{F_2}\alpha_2\rangle_{\delta} \to  \langle \alpha_1 |\Delta F| \alpha_2 \rangle.
\end{align}
Furthermore, uniformly in sufficiently small $\delta>0$, we have the $L^\infty$-estimate 
\begin{align*}
\|\mathbf{X}_{F_0}\|_{\mathrm{op},\delta} := \sup_{0\neq \alpha \in \mathcal{L}} \sqrt{\frac{\langle \mathbf{X}_{F_0} \alpha, \mathbf{X}_{F_0} \alpha\rangle_{\delta} }{\langle  \alpha, \alpha \rangle_{\delta} }} \ll X_0^{1/2+\theta},
\end{align*}
which follows from Lemmas \ref{le:phitrivialbound},  \ref{lem:simplephiasymp}, and Proposition \ref{prop:decayspec}. Then  by \eqref{eq:LinftyL2L2}, which is of course true more generally for bounded operators on an inner product space, we get
\begin{align} \label{eq:introCS}
    |\langle \mathbf{X}_{F_1} \alpha_1, \mathbf{X}_{F_0} \mathbf{X}_{F_2}\alpha_2\rangle_{\delta}| \ll  X_0^{1/2+\theta} \sqrt{  \langle \mathbf{X}_{F_1} \alpha_1, \mathbf{X}_{F_1} \alpha_1\rangle_{\delta}   \langle \mathbf{X}_{F_2} \alpha_2, \mathbf{X}_{F_2}\alpha_2\rangle_{\delta}}.
\end{align}
As $\delta \to 0$ we have
\begin{align*}
    \langle \mathbf{X}_{F_j} \alpha_j, \mathbf{X}_{F_j} \alpha_j\rangle_{\delta} \to \langle \alpha_j | \Delta (F_j \ast F_j) | \alpha_j \rangle,
\end{align*}
and in Lemma \ref{lem:k_Zupperbound} we show that
\begin{align*}
    | (F_j \ast F_j)(\g)| \ll X_j k_{X_j^2,1}(\g).
\end{align*}
Plugging this into (\ref{eq:introCS}) and using \eqref{eq:diraclimit}, we obtain the bound claimed in Theorem \ref{thm:mainblackbox}. 

In the more technical Theorem \ref{thm:technical}, we add Hecke operators that naturally arise when considering a determinant equation $ad-bc=h$ with $|h| > 1$. There we use an $L^\infty$-bound on the spectral side for the Hecke eigenvalues. Note that treating them as a part of the linear functionals is already contained in Theorem \ref{thm:mainblackbox}. Recalling the multiplicativity relations for the Hecke operators, this shows that the determinant $h$ behaves analogously to the scales $X=X_0 X_1 X_2$, the only difference is that for $h$ we cannot choose an arbitrary factorisation but require a factorisation over integers. We also note that for large $h$ we need to normalise the matrices $\smqty(a& b\\ c& d)$ by $1/\sqrt{h}$ to get a weight on $\G$, which replaces $X$ by $X/h$. 

\subsection{Further discussion}
We were partially inspired by the recent work of Pascadi \cite{pascadi2024large}, who showed how the spectral large sieve for exceptional eigenvalues may be bounded in terms of a counting problem in physical space. Indeed, we suspect that Theorem \ref{thm:mainblackbox} could also be proved via Poisson summation and the Kuznetsov trace formula, but this seems circuitous. If required, one can even apply the Poisson summation formula to $a,d$ to write $\Delta k_{X_i^2,R_i}(\tau_1,\tau_2)$ in terms of sums of Kloosterman sums, which shows that we have not lost any information compared to the Kloostermania approach. In Theorem \ref{thm:technical} we use an $L^\infty$ bound for the Hecke eigenvalues, which allows us to capture the advantages of the recent work of Assing, Blomer, and Li \cite{ABL}.

The main benefit compared to previous approaches is that we can carry all information about the functionals $\alpha_1,\alpha_2$ through the Cauchy-Schwarz application, which is especially useful if one has averaging over the level variable. We then get essentially optimal results by bounding the non-negative $ \langle  \alpha_i| \Delta k_{X_i^2,R_i}| \alpha_i\rangle$ crudely by $ \langle  \alpha_i| \mathcal{K} k_{X_i^2,R_i}| \alpha_i\rangle$ and estimating these rather trivially (cf. \cite{GMquadratic} for more details).

As mentioned, we win over trivial upper bounds solely thanks to the spectral gap $\theta  < 1/2$. Remarkably, the applications (Theorems \ref{thm:primefactor}, \ref{thm:roots}, \ref{thm:x2y3}) do not depend at all on the quality of this gap and we get the same results as long as  $\theta  \leq  1/2- \eps $ for some $\eps >0$ independent of $\Gamma$. In particular, the results are independent of progress towards Selberg's eigenvalue conjecture. We expect that this holds true for other applications where one has level averaging, such as primes in arithmetic progressions to large moduli, where also currently the best ranges \cite{pascadi2024large} rely on the Kim-Sarnak bound \cite{KimSarnak}.

A remaining advantage of Kloostermania compared to the results here is in the case when an oscillating weight $F$ causes the spectral expansion to blow-up along the regular spectrum. This happens, for instance, if one considers divisor correlations $d(n)d(n+h)$ with the rough cut-off $n \leq X$ or with a very large shift $h > X^{1+\eps}$, or if one studies the fourth moment of the zeta function in a short interval. In part II, we will bridge this gap by showing how to handle oscillatory weight functions $F$ with at least as good uniformity as one gets from Kloostermania, while keeping all the advantages of this paper. This will be achieved by a more general unskewing procedure that involves both $\a[y]$ and $\n[x]$ matrices in contrast to using only $\a[y]$ matrices.

\section{Background}
This section contains standard background material, much of which is already contained in \cite[Sections 2 and 3]{GM}. 
\subsection{Coordinates}
We start by introducing two coordinate systems for the Lie group 
\begin{align*}
 \G:=  \SL_2(\R)  = \bigg\{ \mqty(a&b \\ c& d) : a,b,c,d \in \R, ad-bc=1  \bigg\} 
\end{align*}
following mostly the notation in \cite{motohashielements}. The relevant subgroups for us are
\begin{align*}
    \mathrm{N} :=& \bigg\{ \mathrm{n}[x] =  \mqty(1&x \\ & 1) : x \in \R  \bigg\}  \\
      \mathrm{A} := & \bigg\{ \mathrm{a}[y]= \mqty(\sqrt{y }& \\ & 1/\sqrt{y}) : y \in (0,\infty)  \bigg\} \\
        \mathrm{K} := &\bigg\{ \mathrm{k}[\theta] =  \mqty(\cos \theta & \sin \theta \\ -\sin \theta & \cos \theta) : \theta \in \R / 2\pi \Z \bigg\}  
\end{align*}
We will refer to the groups $ \mathrm{N}, \mathrm{A}, \mathrm{K}$ and their elements $\n,\a,\k$ without further notice and these symbols are reserved for this purpose.

We then have the \emph{Iwasawa decomposition}
\begin{align*}
    \G= \mathrm{N} \mathrm{A} \mathrm{K}  ,
\end{align*}
where the representation $\g= \smqty(a& b \\ c & d) = \mathrm{n}[x] \mathrm{a}[y] \mathrm{k}[\theta]$ is uniquely defined for $\theta \in [0,2\pi)$ via the sign $\pm$ of $a,b,c,d$ and
\begin{align} \label{eq:matrixtoiwasawa}
   x= \frac{ac+bd}{c^2+d^2},\quad y= \frac{1}{c^2+d^2}, \quad \theta = \arctan \bigg(-\frac{c}{d}\bigg)
\end{align}
and 
\begin{align} \label{eq:iwasawatomatrix}
 \mathrm{n}[x] \mathrm{a}[y] \mathrm{k}[\theta] = \mqty( \sqrt{y} \cos \theta  - \frac{x}{\sqrt{y}} \sin \theta   &  \sqrt{y} \sin \theta  + \frac{x}{\sqrt{y}} \cos \theta \\ - \frac{1}{\sqrt{y}} \sin \theta &\frac{1}{\sqrt{y}} \cos \theta ).
\end{align}
The first two Iwasawa coordinates $(x,y)$ correspond to the upper half-plane $\H:=\{x+iy: y > 0\}$, which is isomorphic to the set of cosets $\G/\mathrm{K}$, in the sense that the group action on $\H$ via $z \mapsto \frac{az+b}{cz+d}$ is the matrix multiplication from the left on $\G/\mathrm{K}$.

We also have the \emph{Cartan decomposition}
\begin{align*}
    \G= \mathrm{K} \mathrm{A} \mathrm{K},
\end{align*}
where the representation $g = \mathrm{k}[\varphi] \mathrm{a}[e^{-\varrho}]\mathrm{k}[\vartheta]$, $\cosh(\varrho) = 2u +1$, is unique provided that $\varrho > 0$ and $\varphi \in [0,\pi)$. We have
\begin{align} \label{eq:cartaninverserho}
    \mathrm{k}[\varphi] \mathrm{a}[e^{\varrho}]\mathrm{k}[\vartheta] = \mathrm{k}[\varphi + \pi/2] \mathrm{a}[e^{-\varrho}]\mathrm{k}[\vartheta - \pi/2] = \mathrm{k}[\varphi + \pi] \mathrm{a}[e^{\varrho}]\mathrm{k}[\vartheta - \pi].
\end{align}
The  Cartan coordinates are best described by the isomorphism of groups via conjugation by the Cayley matrix $\mathcal{C} = \smqty(1 & -i \\ 1 & i)$, which diagonalizes $\mathrm{K}$, 
\begin{align*}
    \SL_2(\R) &\cong \operatorname{SU}_{1,1}(\R) := \bigg\{ \mqty( \alpha & \beta \\ \overline{\beta} & \overline{\alpha}): \alpha,\beta \in \C, \, |\alpha|^2-|\beta|^2 = 1    \bigg\}, \\
    \g &\mapsto \mathcal{C} \g \mathcal{C}^{-1}.
\end{align*}
From this we can compute for $\smqty(a& b \\ c& d)= \mathrm{n}[x] \mathrm{a}[e^t] \mathrm{k}[\theta]=  \mathrm{k}[\varphi] \mathrm{a}[e^{-\varrho}]\mathrm{k}[\vartheta]  $

\begin{align} 
\label{eq:iwasawatocartan}  &\begin{cases} \alpha = \frac{1}{2}(a+d + i(b-c))=e^{i\theta}  (\cosh(t/2) + i e^{-t/2} x /2) = e^{i(\varphi + \vartheta)} \cosh(\varrho/2),  \\ \beta =\frac{1}{2}(a-d - i(b+c) ) =e^{-i\theta} (\sinh(t/2) - i e^{-t/2} x /2) =e^{i(\varphi - \vartheta)} \sinh(-\varrho/2) .\end{cases}
\end{align}
The first two Cartan coordinates $(\varphi,e^{-\varrho})$ correspond to the polar coordinates of the Poincar\'e disk model of the upper half-plane
\begin{align*}
    \H \cong \{ |z| < 1\}: z \mapsto \frac{z-i}{z+i}.
\end{align*}

Denoting 
\begin{align*}
u(z,w):= \frac{|z-w|^2}{4 \Im(z) \Im(w)}, \quad z,w\in \H,    
\end{align*}
 we have $\cosh(\varrho) = 2u(\g i,i) +1$. For $\g= \smqty(a & b \\ c& d) \in \G$ we define
 \begin{align} \label{eq:ulowerbound}
  u(\g) := u(\g i, i) = \tfrac{1}{4}  (a^2+b^2+c^2+d^2-2).
 \end{align}
 We will denote $\mathrm{a}_u = \a[e^{-\varrho}]$ for $\varrho > 0$ such that $\cosh(\varrho) = 2u+1$.

We need the following lemma to pass from Cartan coordinates to Iwasawa coordinates.
\begin{lemma}\label{le:thetaphiSU}
    If $\k[\varphi] \a[e^{-\varrho}] = \k[\varphi] \a_u = \n \a \k[\theta]$, then
\begin{align*}
      e^{i\theta} =e^{i \varphi} \bigg(\frac{\cosh\tfrac{\varrho}{2} + e^{-i2\varphi}\sinh{\tfrac{\varrho}{2}}}{\cosh\tfrac{\varrho}{2} + e^{i2\varphi}\sinh{\tfrac{\varrho}{2}}}\bigg)^{1/2} = e^{i \varphi} \bigg(\frac{\sqrt{1+1/u} + e^{-i2\varphi}}{\sqrt{1+1/u} + e^{i2\varphi}}\bigg)^{1/2},
\end{align*}
with the branch of square-root defined by $1^{1/2} = 1$ and cut along $(-\infty,0]$.
\end{lemma}
\begin{proof}
   Denote Cayley matrix conjugates by
\begin{align*}
 \k^1[\varphi] & :=  \mathcal{C} \k[\varphi] \mathcal{C}^{-1} = \mqty(e^{i\varphi} \\ & e^{-i\varphi}) \\
  \a^1[e^{-\varrho}]&:= \mathcal{C}\a[e^{-\varrho}]\mathcal{C}^{-1}=\mqty(\cosh(\varrho/2) & -\sinh(\varrho/2) \\ -\sinh(\varrho/2) & \cosh(\varrho/2))\\
   \n^1[x]&:= \mathcal{C}\n[x]\mathcal{C}^{-1}=\mqty(1+\frac{ix}{2}& -\frac{ix}{2}\\ \frac{ix}{2} & 1-\frac{ix}{2}).
 \end{align*}
If $\k[\varphi] \a[e^{-\varrho}] = \n \a \k[\theta]$, then $\k^1[\varphi] \a^1[e^{-\varrho}] \k^1[-\theta]= \n^1 \a^1 .$ In particular, denoting by $\diamond$ the action of  $\smqty(a & b \\ c&d)$ on $\C \cup \{\infty\}$ by Möbius transformations $z \mapsto \frac{az +b}{cz+d}$, we have
\begin{align*}
\k^1[\varphi] \a^1[e^{-\varrho}]\k^1[-\theta]\diamond 1 =   \n^1 \a^1\diamond 1 = 1.
\end{align*}
Therefore, we have $ \k^1[-\theta] \diamond 1 = \a^1[e^{\varrho}] \k^1[-\varphi] \diamond 1$, so that
\begin{align*}
    e^{-i2\theta} = \k^1[-\theta] \diamond 1 =   \a^1[e^{\varrho}] \k^1[-\varphi] \diamond 1 = \frac{e^{-i\varphi}\cosh\tfrac{\varrho}{2} + e^{i\varphi}\sinh{\tfrac{\varrho}{2}}}{e^{i\varphi}\cosh\tfrac{\varrho}{2} + e^{-i\varphi}\sinh{\tfrac{\varrho}{2}}} = e^{-2i \varphi} \frac{\cosh\tfrac{\varrho}{2} + e^{i2\varphi}\sinh{\tfrac{\varrho}{2}}}{\cosh\tfrac{\varrho}{2} + e^{-i2\varphi}\sinh{\tfrac{\varrho}{2}}}.
\end{align*}
The lemma follows by multiplying by $e^{2i\varphi}$  and taking the the square-root, where the sign is determined by continuity from the sign at $\varphi= \theta = 0$. This is justified since by $ \Re (\sqrt{1+1/u} + e^{-i2\varphi}) > 0$ we have
\begin{align*}
\frac{\cosh\tfrac{\varrho}{2} + e^{i2\varphi}\sinh{\tfrac{\varrho}{2}}}{\cosh\tfrac{\varrho}{2} + e^{-i2\varphi}\sinh{\tfrac{\varrho}{2}}} = \frac{\sqrt{1+1/u} + e^{-i2\varphi}}{\sqrt{1+1/u} + e^{i2\varphi}}\in \C \setminus  (-\infty,0].
\end{align*}
\end{proof}

Finally, we define the \emph{Bruhat coordinates} which paramterise the big cell $c > 0$
\begin{align} \label{eq:bruhat}
    \g= \mqty(a& b \\ c& d) = \n[r_1] \mathrm{w} \a[c^2] \n[r_2], \quad r_1= \frac{a}{c}, \quad r_2=\frac{d}{c}, \quad \mathrm{w} = \mathrm{k}[\pi/2].
\end{align}

\subsection{Differential operators}
Following the notation in \cite{motohashielements},
we introduce matrices
\begin{align*}
    \mathrm{X}_1 := \mqty(\,\,&1 \\ \,\, &),  \quad \mathrm{X}_2 := \mqty(1& \\ &-1), \quad  \mathrm{X}_3 := \mqty( & 1 \\ -1 & ),
\end{align*}
which are the infinitesimal generators for the subgroups
\begin{align*}
    \mathrm{N} = \{\exp( t \mathrm{X}_1): t \in \R\}, \,    \mathrm{A} = \{\exp( t \mathrm{X}_2): t \in \R\}, \, \mathrm{K} = \{\exp( t \mathrm{X}_3): t \in \R\}.
\end{align*}
We then define the right Lie differentials
\begin{align*}
    \mathrm{x_j}f(\g) := \partial_t f(\g \,\exp(t \mathrm{X}_j)) |_{t=0}.
\end{align*}
By construction these operators are left-invariant under the action of the group $\G$ on itself. That is, denoting
\begin{align}\label{eq:lrdef}
    l_\h: f(\g) \mapsto f(\h\g), \quad r_\h: f(\g) \mapsto f(\g \h),
\end{align}
we have $\mathrm{x}_j l_\h=l_\h\mathrm{x}_j$ for $j\in\{1,2,3\}$.

We construct the \emph{raising and lowering operators}
\begin{align*}
    \mathrm{e}^+ :=& 2 \mathrm{x}_1 + \mathrm{x}_2-i\mathrm{x}_3, \\
      \mathrm{e}^- :=& -2i \mathrm{x}_1 + \mathrm{x}_2+i\mathrm{x}_3.
\end{align*}
In the Iwasawa coordinates  $ \mathrm{n}[x] \mathrm{a}[y] \mathrm{k}[\theta]$ we have
\begin{align*}
\mathrm{e}^+  &= e^{2i\theta} (2iy\partial_x + 2 y \partial_y-i\partial_\theta),\\
\mathrm{e}^- &= e^{-2i\theta}(-2iy\partial_x + 2 y \partial_y+i\partial_\theta),\\
\mathrm{x}_3 &= \partial_\theta.
\end{align*}
The operators $\mathrm{e}^+,\mathrm{e}^-$ correspond to $2 K_k, 2 \Lambda_k$ in the notations of \cite[Section 4]{DFIartin}. Denoting the commutator of differential operators by $[\mathrm{x},\mathrm{y}] = \mathrm{x} \mathrm{y} - \mathrm{y} \mathrm{x}$, we have by \cite[(7.2)]{motohashi}
\begin{align}\label{eq:diffopcommutator}
[\mathrm{x}_3, \mathrm{e}^{\pm}] =  \pm 2i \mathrm{e}^{\pm}, \quad [\mathrm{e}^+, \mathrm{e}^-] = -4i \mathrm{x}_3.
\end{align}

The \emph{Casimir operator} is then defined as
\begin{align} \label{eq:casimirdef}
    \Omega := -\frac{1}{4}  \mathrm{e}^{+}  \mathrm{e}^{-} + \frac{1}{4}\mathrm{x}_3^2 - \frac{1}{2} i \mathrm{x}_3 =  -\frac{1}{4}   \mathrm{e}^{-}\mathrm{e}^{+}  + \frac{1}{4}\mathrm{x}_3^2 + \frac{1}{2} i \mathrm{x}_3.
\end{align}
The Casimir element is not only left-invariant but also right-invariant under the action of the group $G$, that is,
\begin{align}\label{eq:Omegalrcommute}
    \Omega l_\h =  l_\h \Omega, \quad  \Omega r_\h =  r_\h \Omega.
\end{align}
In the Iwasawa coordinates $ \mathrm{n}[x] \mathrm{a}[y] \mathrm{k}[\theta]$ the Casimir operator is given by
\begin{align} \label{eq:casimiriwasawa}
    \Omega = -y^2 (\partial_x^2 + \partial_y^2) + y \partial_x \partial_\theta,
\end{align}
where the first part $-y^2 (\partial_x^2 + \partial_y^2)$ is the Laplace operator on the upper half-plane $\H$.
In the Cartan coordinates $\mathrm{k}[\varphi] \mathrm{a}[e^{-\varrho}]\mathrm{k}[\vartheta]$, denoting  $\cosh \varrho =2u+1,$ we have
\begin{align} \label{eq:casimircartan}
\begin{split}
    \Omega =&  - \partial_\varrho^2 - \frac{1}{\tanh \varrho} \partial_\varrho -\frac{1}{4\sinh^2 \varrho} \partial_\varphi^2  +\frac{1}{2 \sinh\varrho \tanh \varrho} \partial_\varphi \partial_\vartheta - \frac{1}{4 \sinh^2 \varrho}  \partial_\vartheta^2 \\
        =& -u(u+1) \partial_u^2 - (2u+1) \partial_u - \frac{1}{16u(u+1)}  \partial_\varphi^2   +   \frac{2u+1}{8 u(u+1)}\partial_\varphi \partial_\vartheta -  \frac{1}{16u(u+1)} \partial_\vartheta^2 .
\end{split}
\end{align}
Define the operator $\iota $ by $\iota f (\g) = f(\g^{-1})$. By the Bruhat decomposition \eqref{eq:bruhat} and by \eqref{eq:matrixtoiwasawa} we have
\begin{align} \label{eq:cartantobruhat}
 \partial_\vartheta = \partial_\theta = \frac{\partial r_2}{\partial \theta} \partial_{r_{2}} = (r_2^2+1) \partial_{r_2},   \quad \quad   \partial_\varphi = -\iota \partial_\vartheta \iota  =  (r_1^2+1) \partial_{r_1}.
\end{align}

\subsection{Integration}

The invariant Haar measure $\d \g$ on $\G$ is defined in terms of the Iwasawa and Cartan coordinates via (note the swap between $\a$ and $\n$)
\begin{align} \label{eq:integrationonG}
\begin{split}
     \d \mathrm{n}[x] = \d x, \quad & \d \mathrm{a}[y] = \frac{\d y}{y}, \quad \d \mathrm{k}[\theta] = \frac{\d \theta}{2\pi}, \\
   \int_\G  f(\g)\d \g =& \int_{\mathrm{A} \times \mathrm{N} \times \mathrm{K}} f(\a \n \k)  \d \a\d \n \d \k  =    \int_{\R\times \R_{>0} \times \R/2\pi\Z} f(\n[x] \a[y] \k[\theta]) \frac{\d x \d y \d \theta}{2\pi y^2} \\
   =&\int_{\mathrm{K} \times \mathrm{K}} \int_{\R_{>0}} f(\k_1 \a[e^{-\varrho}] \k_2)  \sinh(\varrho) \d \varrho \,\d \k_1 \d \k_2 \\
   =& 2 \int_{\mathrm{K} \times \mathrm{K}} \int_{\R_{>0}} f(\k_1 \a_u \k_2)   \d u \,\d \k_1 \d \k_2 
\end{split}
\end{align}
We have for any integrable $F:\G \to\C$
\begin{align} \label{eq:integralinmatrix}
   \frac{1}{ |\Gamma_0(q)\backslash \G|}    \int_\G  F(\g)\d \g = \frac{1}{  \zeta(2) \, q \, \prod_{p|q}(1+p^{-1})}   \int_{\R^3}  F(\smqty(a & \ast \\ c & d)) \frac{\d a \d c \d d}{c}.
\end{align}

The inner product on $L^2(\G)$ is defined by
\begin{align*}
    \langle f_1,f_2\rangle_\G := \int_\G f_1(\g) \overline{f_2(\g)} \d \g.
\end{align*}
Then we have the integration by parts formula
\begin{align} \label{eq:integrationbypartsraisinglowering}
   \langle 
 \mathrm{e}^{\pm} f_1,f_2\rangle_\G  = -   \langle 
 f_1, \mathrm{e}^{\mp}f_2\rangle_\G
\end{align}
 and $\Omega$ is  a symmetric operator, that is, 
 \begin{align} \label{eq:omegaselfadjoint}
      \langle 
\Omega f_1,f_2\rangle_\G   =   \langle 
 f_1, \Omega f_2\rangle_\G. 
 \end{align}

\subsection{Projections to fixed types}
For $\ell_1,\ell_2 \in \Z$  we say that a function $f:\G \to\C$ has left type $\ell_1$ (resp. right type $\ell_2$) if for all $\k[\theta]$ we have
\begin{align*}
    f(\k[\theta] \g ) = f(\g) e^{i\ell_1 \theta}. \quad (\text{resp. }  f(\g \k[\theta]) = f(\g) e^{ i\ell_2 \theta}).
\end{align*}
For any integrable $f:\G\to\C$ and any $\ell_1,\ell_2 \in \Z$ we define the projections to a fixed type by
\begin{align*}
(\mathcal{F}_{\ell_1,\ell_2} f) (\g) := \int_{\K \times \K} f(\k[\theta_1] \g \k[\theta_2])  e^{-i \ell_1 \theta_1 - i \ell_2 \theta_2} \d \k[\theta_1] \d \k[\theta_2].
\end{align*}
By \eqref{eq:cartaninverserho} we have
\begin{align*}
    (\mathcal{F}_{\ell_1,\ell_2} f) (\g) =   (\mathcal{F}_{\ell_1,\ell_2} f) (\k[-\pi]\g\k[\pi]) =  (-1)^{\ell_2-\ell_1}(\mathcal{F}_{\ell_1,\ell_2} f) (\g),
\end{align*}
which implies that $  \mathcal{F}_{\ell_1,\ell_2} f$ vanishes unless $\ell_1\equiv \ell_2 \pmod{2}$.  We then have for any smooth compactly supported $f:\G \to \C$ the Fourier series expansion
\begin{align} \label{eq:fouriertypeexpansion}
    f(\g) = \sum_{\substack{\ell_1,\ell_2 \in \Z \\ \ell_1 \equiv \ell_2 \pmod{2}}}  (\mathcal{F}_{\ell_1,\ell_2} f)(\g).
\end{align}

Let $\ell \in \Z$ and $\nu \in \C$. The basic functions (defined using the Iwasawa coordinates)
\begin{align} \label{eq:phibasicdefinition}
    \phi_{\ell}(\g,\nu):=y^{1/2+\nu} e^{ i\ell \theta}
\end{align}
 are  eigenfunctions of the Casimir operator $\Omega$ on $\G$ of right-type $\ell$ with spectral parameter $\nu$, that is
\begin{align*}
    \Omega   \phi_{\ell}(\g,\nu) = \Bigl(\frac{1}{4}-\nu^2\Bigr) \phi_{\ell}(\g,\nu).
\end{align*}
For $\ell_1 \equiv \ell_2 \pmod{2}$ we define their projections to a fixed left-type as (recall \eqref{eq:cartaninverserho})
\begin{align} \label{eq:2pitopi}
    \phi_{\ell_1,\ell_2}(\g,\nu) &:= \frac{1}{2 \pi} \int_0^{2\pi}   \phi_{\ell_2}(\k[\varphi]\g,\nu) e^{-i\ell_1 \varphi} \d \varphi = \frac{1}{ \pi} \int_0^{\pi}   \phi_{\ell_2}(\k[\varphi]\g,\nu) e^{-i\ell_1 \varphi} \d \varphi. 
\end{align}

\subsection{Spectrum of $L^2(\Gamma\backslash \G,\chi)$}
We denote by $\mathfrak{C}(\Gamma,\chi)$ a set of inequivalent representatives for the singular cusps of $\Gamma \setminus \G$ for the character $\chi$. We let $\kappa \in \{0,1\}$ denote the parity of $\chi$, that is, $\chi(\smqty(-1 & \\ & -1)) = (-1)^\kappa$. The notion of singular cusps and scaling  matrices coincide, using the identification $\G/\mathrm{K} \cong \H$, with the upper half-plane setting (see \cite[Section 4.1]{Drappeau}, for instance). 
 Let $ \phi_{\ell}(g,\nu)$ be as  in \eqref{eq:phibasicdefinition} and suppose that $\ell \equiv \kappa \pmod{2}$. For $\mathfrak{c}\in \mathfrak{C}(\Gamma,\chi)$, we define the Eisenstein series for $\R(\nu)>1/2$ by
\begin{align*}
    E_{\mathfrak{c}}^{(\ell)}(\g,\nu)=\sum_{\gamma\in \Gamma_{\mathfrak{c}}\setminus \Gamma } \overline{\chi}(\gamma) \phi_\ell (\sigma_{\mathfrak{c}}^{-1} \gamma \g, \nu),
\end{align*}
where $\Gamma_{\mathfrak{c}}$ is the stabilising group and $\sigma_{\mathfrak{c}}$ a scaling matrix of the cusp $\mathfrak{c}$. The Eisenstein series are extended to $\nu \in \C$ by meromorphic continuation, and have only one simple pole at $\nu=1/2$ for congruence subgroups $\Gamma$ (cf. for instance, \cite[Theorem 11.3]{IwaniecBook}).
\begin{proposition}\label{prop:specdecompL2}\emph{(Spectral expansion of $L^2(\Gamma \backslash \G,\chi)$)}. Let $\Gamma$ be a congruence subgroup and $\chi$ be a group character for $\Gamma$ of parity $\kappa \in \{0,1\}$. There exists a countable set  $\mathcal{B}(\Gamma,\chi) $ and for $V \in \mathcal{B}(\Gamma,\chi),\ell\in \Z, \ell \equiv \kappa \pmod{2}$, complex numbers $\nu_V$ and smooth functions $\varphi^{(\ell)}_V \in L^2(\Gamma \backslash \G,\chi)$ of right-type $\ell$ with \begin{align}\label{eq:varphieigenfuncOmega}
        \Omega \varphi^{(\ell)}_V=\bigg(\frac{1}{4}-\nu_V^2\bigg)\varphi^{(\ell)}_V
    \end{align} 
    such that the following holds. For any smooth and bounded $f\in L^2(\Gamma \backslash \G,\chi)$ with $\Omega f$ bounded we have
    \begin{align*}
        f(\g)= &\mathbf{1}_{\chi \,\mathrm{principal}}\frac{1}{|\Gamma \backslash G|} \langle f ,1 \rangle_{\Gamma \backslash G} +\sum_{V \in \mathcal{B}(\Gamma,\chi)} \sum_{\ell\equiv \kappa \pmod{2}}\langle f ,\varphi_V^{(\ell)} \rangle_{\Gamma \backslash G}\,  \varphi_V^{(\ell)}(\g)\\
        &+\sum_{\mathfrak{c}\in \mathfrak{C}(\Gamma,\chi)} \sum_{\ell\equiv \kappa \pmod{2}} \frac{1}{4 \pi i} \int_{(0)} \langle f, E^{(\ell)}_{\mathfrak{c}}(\cdot,\nu) \rangle_{\Gamma \backslash G} \, E^{(\ell)}_{\mathfrak{c}}(\g,\nu) \d \nu.
    \end{align*}
\end{proposition}
Note that while the Eisenstein series are not square-integrable, for $f \in L^{2}(\Gamma \backslash \G,\chi)$ the inner product  $\langle f, E^{(\ell)}_{\mathfrak{c}}(*,\nu) \rangle_{\Gamma \backslash G} $ exists and is finite.
For $\Gamma= \Gamma_0(q)$  Proposition \ref{prop:specdecompL2} follows from \cite[Section 4]{DFIartin}, by identifying for a fixed $\ell$ the basis elements $u_j$ of weight $k_{\text{DFI}}=\ell$ in  \cite[(4.50))]{DFIartin} via
\begin{align*}
  \varphi_V^{(\ell)}(\n[x] \a[y] \k[\theta])=  u_j(x+iy) e^{i\ell \theta},
\end{align*}
and similarly for the Eisenstein series.  The proof for general congruence subgroups follows along similar lines.

We now describe the functions $\varphi_V^{(\ell)}$ in the discrete spectrum. The objects $V$ appearing are the irreducible subspaces of the cuspidal part of $L^2(\Gamma\backslash \G,\chi)$, and the Casimir operator is constant on each $V$, that is, there is some $\nu_V$ such that for all smooth $f \in V$ we have $\Omega f = (\frac{1}{4}-\nu_V^2) f$. Each $V$ splits into subspaces according to the right-action by $\mathrm{K}$
\begin{align*}
    V = \bigoplus_{\ell \equiv \kappa \pmod{2}} V^{(\ell)},
\end{align*}
where $V^{(\ell)}$ is a one or zero dimensional subspace consisting of functions of right-type $\ell$. The raising and lowering operators define maps $\mathrm{e}^{\pm}: V^{(\ell)} \to  V^{(\ell \pm 2)}$. We pick generators $\varphi^{(\ell)}_V$ for each $V^{(\ell)}$ such that
\begin{align*}
    \langle \varphi^{(\ell)}_V, \varphi^{(\ell)}_V \rangle_{\Gamma \backslash G} =1
\end{align*}
to get an orthonormal basis.
 
The cuspidal spectrum $\mathcal{B}(\Gamma,\chi)$ splits into three parts
\begin{enumerate}
    \item Principal series: $\nu_V \in i \R$
    \item Exceptional spectrum: $\nu_V \in (0,1/2)$
    \item Discrete series: $\nu_V =  \frac{k-1}{2}, \quad k \geq 2, \quad k \equiv \kappa \pmod{2}$.
    \end{enumerate}
We call the union of the principal and discrete series the regular spectrum. 
 The Selberg eigenvalue conjecture states that the exceptional spectrum is empty. The best bound towards this is by Kim-Sarnak \cite{KimSarnak}, which states that for congruence subgroups we have $\nu_V \in (0,7/64)$ in the possible exceptional spectrum. 
 
 In the first two cases the $V^{(\ell)}$ are non-trivial for all $\ell \in \Z$. For the discrete series $\nu_V = (k-1)/2$ we call $k$ the \emph{weight} and we have either
\begin{align*}
    V = \bigoplus_{\substack{\ell  \geq k\\ \ell \equiv \kappa\pmod{2}}} V^{(\ell)} \quad \text{or} \quad   V= \bigoplus_{\substack{\ell  \leq - k\\  \ell \equiv \kappa\pmod{2}}} V^{(\ell)}.
\end{align*}
In either case the edge function $\varphi_V^{(\pm k)}$ is annihilated by the lowering/raising operator, that is, $\mathrm{e}^{\mp} \varphi_V^{(\pm k)} = 0$.

\subsection{Hecke operators}

For $h$ coprime to $q$ we then define the Hecke operator acting on $f:\G \to \C$ is defined by
\begin{align} \label{eq:heckeopdef}
    \mathcal{T}_h f(g) :=  \frac{1}{\sqrt{h}} \sum_{ad=h} \chi(a) \sum_{b  \pmod{d}} f\bigg(\frac{1}{\sqrt{h}}\mqty(a & b  \\ & d)g\bigg).
\end{align} 
Then for $\gcd(h,q)=1$ the Hecke operators $\mathcal{T}_h$ commute with each other and $\Omega$,  and furthermore are normal (since $\langle \mathcal{T}_h f,g \rangle = \langle  f, \overline{\chi}(h)\mathcal{T}_h g \rangle $), so that we can choose a common orthonormal basis. Then 
\begin{align}\label{eq:heckecusp}
    \mathcal{T}_h \varphi_V^{(\ell)} (\tau) = \lambda_V(h) \varphi_V^{(\ell)} (\tau)
\end{align}
for the Hecke eigenvalues $\lambda_V(h)$. Similarly, for the Eisenstein series we have for $\gcd(h,q)=1$ \cite[(6.16)]{DFIartin}
\begin{align}\label{eq:heckeeis}
    \mathcal{T}_h E_{\mathfrak{c}}^{(\ell)} (\tau,\nu) = \lambda_{\mathfrak{c},\nu}(h) E_{\mathfrak{c}}^{(\ell)} (\tau,\nu)
\end{align}
with $\lambda_{\mathfrak{c},\nu}(h) $ given explicitly in \cite[(6.17)]{DFIartin}. 
We have  $|\lambda_{\mathfrak{c},\nu}(h)| \leq d(h)  \ll_\eps |h|^{o(1)}.$ The Ramanujan-Petersson conjecture states that also in the cuspidal spectrum
\begin{align*}
 |\lambda_{V}(h)| \ll |h|^{o(1)}.
\end{align*}
For the discrete series, this was proved by Deligne \cite{Deligne1968-1969}. In general the best bound towards this is $ \ll |h|^{\vartheta+o(1)}$ with $\vartheta\leq 7/64$ by Kim and Sarnak \cite{KimSarnak}. 

We have the Rankin-Selberg bound, which gives the Ramanujan-Peterson conjecture on average (cf.  for instance, \cite[Theorem 8.3]{IwaniecBook} for $\Gamma$ of level $q$)
    \begin{align} \label{eq:rankinselberg}
        \sum_{\substack{h \leq H \\ (h,q)=1}} |\lambda_V(h)|^2 \leq ((1+|\nu_V|)q)^{o(1)} H.
    \end{align}

\subsection{Unitary normalisation of $\phi_{\ell_1,\ell_2}(\g,\nu)$} 
In this subsection we define the functions $\mathcal{P}^{(\ell_1,\ell_2)}_{\nu}$ (see  \eqref{eq:unitarynorm}). They are normalised versions of $\phi_{\ell_1,\ell_2}(\cdot,\nu)$ and will appear in the form \eqref{eq:ratioPphi} in Section \ref{sec:specautokern} in the proof of the spectral expansion of an automorphic kernel. As will be seen in Lemma \ref{le:orthogonpol}, for the discrete series they correspond to the unitarily normalised harmonics of the ambient space $\G=\SL_2(\R)$.

In the notations of \cite[(4.57), (4.58)]{DFIartin} we have the correspondence 
\begin{align}\label{eq:DFIcorrespondence}
    \varphi_V^{(\ell)}(\n[x] \n[y] ) = u_{j \ell}(x+iy) \quad \text{or} \quad u_{j k \ell}(x+iy),
\end{align}
where $j$ parametrises the $V$ that appear in the spectrum and $k$ denotes the minimal type of $V$ in the case of the discrete series. Then by \cite[Corollary 4.4]{DFIartin} we can choose the orthonormal basis so that for $k=k_V$ and $\ell \geq k, \ell \equiv k \pmod{2}$
\begin{align} \label{eq:raisingvarphi}
     \varphi_V^{(\pm \ell)} = G(\nu_V,\ell) (e^{\pm})^{(\ell-k)/2}\varphi_V^{\pm k},
\end{align}
with (see  \cite[eq. (4.59), (4.60)]{DFIartin})
\begin{align*}
    G(\nu,\ell) = \begin{dcases}
        2^{-(\ell-\kappa)/2} \bigg| \frac{\Gamma(\nu +\frac{1+\kappa}{2}) \Gamma(-\nu + \frac{1+\kappa}{2})}{\Gamma(\nu+\frac{1+\ell}{2}) \Gamma(-\nu+\frac{1+\ell}{2})} \bigg|^{1/2}, \quad &\Re(\nu) \in [0,1/2) \\
        2^{-(\ell-k)/2}\bigg|\frac{\Gamma(k)}{\Gamma(\frac{\ell+k}{2})\Gamma(\frac{\ell-k}{2}+1)} \bigg|^{1/2}, \quad &\nu= \frac{k-1}{2}.
    \end{dcases}
\end{align*}

For the basic functions given by \eqref{eq:phibasicdefinition}, we have have for any $\ell \geq \kappa$ with $\ell \equiv \kappa \pmod{2}$ by induction
\begin{align} \label{eq:raisingphi}
\phi_{\pm  \ell}(g,\nu)  := G_\phi(\nu,\ell) (\mathrm{e}^{\pm})^{ (\ell-\kappa)/2} \phi_{\pm \kappa}(g,\nu), \quad \quad   G_\phi(\nu,\ell) := 2^{-(\ell-\kappa)/2}\frac{ \Gamma(\frac{1+\kappa}{2}+\nu)} {\Gamma(\frac{1+\ell}{2}+\nu)}.
\end{align}
Then for $\nu \in i \R$ we have $|G(\nu,\ell)| = |G_\phi(\nu,\ell)|$. 

For any even $\mu >0$ and $\ell \geq 0$ (whenever $\varphi_V^{(\pm \ell)}\neq 0$)
\begin{align} \label{eq:upanddown}
\begin{split}
      \frac{(\mathrm{e}^{\mp})^{\mu/2}   (\mathrm{e}^{\pm})^{\mu/2} \varphi_V^{(\pm \ell)}(\g)}{\varphi_V^{(\pm \ell)}(\g)}   =&     \frac{(\mathrm{e}^{\mp})^{\mu/2}   (\mathrm{e}^{\pm})^{\mu/2} \phi_{\pm \ell}(\g,\nu_V)}{\phi_{\pm \ell}(\g,\nu_V)},
\end{split}
\end{align}
which can be seen from  \eqref{eq:diffopcommutator} and \eqref{eq:casimirdef} by induction.

For $|\ell_1|,|\ell_2| \geq k_\nu$ we define the functions  (whenever $\varphi^{(\ell_1)}_V\varphi^{(\ell_2)}_V \neq 0$)
\begin{align} \label{eq:unitarynorm}
\mathcal{P}^{(\ell_1,\ell_2)}_{\nu}(\g) &:= \frac{G_\phi(\nu,|\ell_1|)}{G_{\phi}(\nu,|\ell_2|)} \frac{G(\nu,|\ell_2|)}{G(\nu,|\ell_1|)} \phi_{\ell_1,\ell_2}(\g, \nu)
\end{align}
Denoting $\mathrm{e}^m=(\mathrm{e}^{\sgn(m)})^m$,  by \eqref{eq:raisingvarphi} we see that
\begin{align} \label{eq:Gtwoells}
G(V,\ell_1,\ell_2) :=  \frac{ \varphi_{V}^{(\ell_2)}(\g)}{ \mathrm{e}^{(\ell_2-\ell_1)/2}   \varphi_{V}^{(\ell_1)}(\g)}, \quad G_\phi(\nu,\ell_1,\ell_2) :=  \frac{ \phi_{\ell_2}(\g,\nu)}{ \mathrm{e}^{(\ell_2-\ell_1)/2}   \phi_{\ell_1}(\g,\nu)}   
\end{align}
do not depend on $\g$. We claim that (whenever $\varphi^{(\ell_1)}_V\varphi^{(\ell_2)}_V \neq 0$)
\begin{align} \label{eq:ratioPphi}
\mathcal{P}^{(\ell_1,\ell_2)}_{\nu_V}(\g) =& \frac{G(V,\ell_1,\ell_2) }{G_\phi(\nu,\ell_1,\ell_2)} \phi_{\ell_1,\ell_2}(\g, \nu_V).
\end{align}
To see this, we substitute  \eqref{eq:raisingphi} and \eqref{eq:raisingvarphi}  to get (note that one part involves $\kappa$, the other one $k$)
\begin{align*}
\frac{\mathrm{e}^{(\ell_2-\ell_1)/2} \phi_{\ell_1}(\cdot,\nu_V)}{\phi_{\ell_2}(\cdot,\nu_V) } & \frac{\varphi^{(\ell_2)}_V(\cdot)}{\mathrm{e}^{(\ell_2-\ell_1)/2}\varphi^{(\ell_1)}_V(\cdot) } \\
    &= \frac{G_\phi(\nu_V,|\ell_1|) }{G_\phi(\nu_V,|\ell_2|)} \frac{G(\nu_V,|\ell_2|)}{G(\nu_V,|\ell_1|) }\frac{\mathrm{e}^{(\ell_2-\ell_1)/2}  \mathrm{e}^{(\ell_1-\kappa)/2} \phi_{\kappa}(\cdot,\nu_V)} {\mathrm{e}^{(\ell_2-\kappa)/2}\phi_{\kappa}(\cdot,\nu_V)} \frac{\mathrm{e}^{(\ell_2-k)/2}\varphi_{V}^{(k)}(\cdot)}{\mathrm{e}^{(\ell_2-\ell_1)/2}  \mathrm{e}^{(\ell_1-k)/2} \varphi_{V}^{(k)}(\cdot)}.
\end{align*}
If $|\ell_2| \geq |\ell_1|$ and $\sgn(\ell_1)=\sgn(\ell_2)$, then we can combine $\mathrm{e}^{(\ell_2-\ell_1)/2}  \mathrm{e}^{(\ell_1-\kappa)/2} = \mathrm{e}^{(\ell_2-\kappa)/2}$ and $\mathrm{e}^{(\ell_2-\ell_1)/2}  \mathrm{e}^{(\ell_1-k)/2} = \mathrm{e}^{(\ell_2-k)/2}$ to get \eqref{eq:ratioPphi}. If $|\ell_2|< |\ell_1|$ and $\sgn(\ell_1)=\sgn(\ell_2)$,  we can write  \begin{align*} 
\mathrm{e}^{(\ell_2-\ell_1)/2}  \mathrm{e}^{(\ell_1-\kappa)/2} = \mathrm{e}^{(\ell_2-\ell_1)/2}  \mathrm{e}^{(\ell_1-\ell_2)/2} \mathrm{e}^{(\ell_2-\kappa)/2}, \\
\mathrm{e}^{(\ell_2-\ell_1)/2}  \mathrm{e}^{(\ell_1-k)/2} = \mathrm{e}^{(\ell_2-\ell_1)/2}  \mathrm{e}^{(\ell_1-\ell_2)/2} \mathrm{e}^{(\ell_2-k)/2},
\end{align*}
and use \eqref{eq:upanddown} with $\mu=|\ell_1-\ell_2|$ to see that again \eqref{eq:ratioPphi} holds. Finally, if $\sgn(\ell_2) \neq \sgn(\ell_1)$, then we can write
 \begin{align*} 
\mathrm{e}^{(\ell_2-\ell_1)/2}  \mathrm{e}^{(\ell_1-\kappa)/2} = \mathrm{e}^{(\ell_2-\kappa)/2}  \mathrm{e}^{(\kappa-\ell_1)/2} \mathrm{e}^{(\ell_1-\kappa)/2},
\end{align*}
and use \eqref{eq:upanddown} with $\mu=|\ell_1-\kappa|$ to see that again \eqref{eq:ratioPphi} holds, recalling that the case $\sgn(\ell_2) \neq \sgn(\ell_1)$ only occurs for $k_V=\kappa$.

Finally, we calculate for $\ell_1,\ell_2 \geq 0$ 
\begin{align} \label{eq:Gfactor}
\begin{split}
     \bigg| \frac{G_\phi(\nu,\ell_1) }{G_\phi(\nu,\ell_2)} \frac{G(\nu,\ell_2)}{G(\nu,\ell_1) } \bigg| & = \bigg| \frac{\Gamma(\nu + \frac{1+\ell_2}{2})\Gamma(-\nu + \frac{1+\ell_1}{2})}{\Gamma(\nu + \frac{1+\ell_1}{2})\Gamma(-\nu + \frac{1+\ell_2}{2})}\bigg|^{1/2}.
\end{split}
\end{align}
For $\nu\in i\R$
\begin{align}\label{eq:Gfactorregular}
     \bigg| \frac{G_\phi(\nu,\ell_1) }{G_\phi(\nu,\ell_2)} \frac{G(\nu,\ell_2)}{G(
    \nu,\ell_1) } \bigg| =1,
\end{align}
for $\nu= \frac{k-1}{2}$ and $\ell_1,\ell_2 \geq k$ we get
\begin{align}\label{eq:Gfactorholo}
     \bigg| \frac{G_\phi(\nu,\ell_1) }{G_\phi(\nu,\ell_2)} \frac{G(\nu,\ell_2)}{G(\nu,\ell_1) } \bigg| = \bigg( \frac{\Gamma( \frac{k+\ell_2}{2})\Gamma(1+ \frac{\ell_1-k}{2})}{\Gamma( \frac{k+\ell_1}{2})\Gamma(1+ \frac{\ell_2-k}{2})}\bigg)^{1/2},
\end{align}
and for $\nu \in [0,1/2)$ we have
\begin{align}\label{eq:Gfactorexcep}
     \bigg| \frac{G_\phi(\nu,\ell_1) }{G_\phi(\nu,\ell_2)} \frac{G(\nu,\ell_2)}{G(\nu,\ell_1) } \bigg| \asymp (1+\ell_1)^{2\nu}(1+\ell_2)^{2\nu}.
\end{align}

\section{Spectral expansion of an automorphic kernel} \label{sec:specautokern}
The proof of Theorem \ref{thm:mainblackbox} is based on applying the spectral expansion of an automorphic kernel, the Cauchy Schwarz inequality, and reversal of the spectral expansion. In particular, we need to be able to do this for general kernel functions involving different left and right types. 

For a smooth compactly supported $F:\G\to\C$, congruence subgroup $\Gamma$,  and a  group character $\chi$ of parity $\kappa \in \{0,1\}$, recall the construction of an automorphic kernel
\begin{align*}
    ( \mathcal{K} F)(\tau_1,\tau_2) =     ( \mathcal{K}_{\Gamma,\chi} F)(\tau_1,\tau_2) = \sum_{\gamma \in \Gamma} \overline{\chi}(\gamma)  F(\tau_1^{-1} \gamma \tau_2).
\end{align*}
Using the Fourier expansion \eqref{eq:fouriertypeexpansion} we can write
\begin{align*}
    ( \mathcal{K} F)(\tau_1,\tau_2) = \sum_{\substack{\ell_1,\ell_2 \in \Z \\ \ell_1 \equiv \ell_2  \pmod{2}}}   ( \mathcal{K} \mathcal{F}_{\ell_1,\ell_2}F)(\tau_1,\tau_2). 
\end{align*}
For any $\gamma \in \Gamma$ we have
\begin{align*}
   ( \mathcal{K} F)(\tau_1,\gamma\tau_2) = \chi(\gamma) ( \mathcal{K} F)(\tau_1,\tau_2) =   ( \mathcal{K} F)(\gamma^{-1}\tau_1,\tau_2).
\end{align*}
Thus, for $\ell_1 \equiv \ell_2 \not \equiv \kappa \pmod{2}$ we see by the action of $\k[\pi] =\smqty(-1 & \\ & -1) \in \Gamma$ with $\chi(\smqty(-1 & \\ & -1)) = (-1)^\kappa$ that
\begin{align*}
 ( \mathcal{K} \mathcal{F}_{\ell_1,\ell_2}F)(\tau_1,\tau_2) =  & \frac{1}{2} \bigg( ( \mathcal{K} \mathcal{F}_{\ell_1,\ell_2}F)(\tau_1,\tau_2)  +  (-1)^\kappa  ( \mathcal{K} \mathcal{F}_{\ell_1,\ell_2}F)(\smqty(-1 & \\ & -1) \tau_1,\tau_2))\bigg) \\
    =& \frac{1}{2} \bigg( ( \mathcal{K} \mathcal{F}_{\ell_1,\ell_2}F)(\tau_1,\tau_2)  -  (-1)^{2\kappa}  ( \mathcal{K} \mathcal{F}_{\ell_1,\ell_2}F)(\tau_1,\tau_2)\bigg)  =0.
\end{align*}
Hence, the Fourier expansion can be restricted to $\ell_1 \equiv \ell_2 \equiv \kappa \pmod{2}$
\begin{align*}
         ( \mathcal{K} F)(\tau_1,\tau_2)  =  \sum_{\substack{\ell_1,\ell_2 \in \Z \\ \ell_1 \equiv \ell_2 \equiv \kappa \pmod{2}}}  (\mathcal{K} \mathcal{F}_{\ell_1,\ell_2}F)(\tau_1,\tau_2).
\end{align*}
The main result of this section is the following spectral expansion for $\mathcal{K} \mathcal{F}_{\ell_1,\ell_2}F$ that involves the functions $\mathcal{P}_\nu^{(\ell_1,\ell_2)}$ defined in \eqref{eq:unitarynorm}. By abuse of notation we will denote $\PP_V^{(\ell_1,\ell_2)} = \PP_{\nu_V}^{(\ell_1,\ell_2)}$.
\begin{proposition} \label{prop:kernelexpansion}
Let $\Gamma$ be a congruence subgroup and $\chi$ be a group character for $\Gamma$ of parity $\kappa \in \{0,1\}$.    Let $F:\G \to \C$ be smooth and compactly supported. Then for $\ell_1\equiv \ell_2 \equiv \kappa \pmod{2}$
    \begin{align*}
 (\mathcal{K} \mathcal{F}_{\ell_1,\ell_2}F)(\tau_1,\tau_2) =&  \frac{\mathbf{1}_{\chi \,\mathrm{principal}} \mathbf{1}_{\ell_1=\ell_2=0}}{|\Gamma \backslash \G|} \int_\G F(g) \d g \\&+ \sum_{V\in \mathcal{B}(\Gamma,\chi)} \langle F, \mathcal{P}^{(\ell_1,\ell_2)}_{V}\rangle_{\G} \overline{\varphi_V^{(\ell_1)}(\tau_1)}\varphi_V^{(\ell_2)}(\tau_2) \\
 &+\sum_{\mathfrak{c}\in \mathfrak{C}(\Gamma,\chi)} \frac{1}{4 \pi i}  \int_{(0)} \langle F, \mathcal{P}^{(\ell_1,\ell_2)}_{\nu}\rangle_{\G} \overline{E^{(\ell_1)}_{\mathfrak{c}}(\tau_1,\nu)}E^{(\ell_2)}_{\mathfrak{c}}(\tau_2,\nu) \d \nu.
    \end{align*}
\end{proposition}
The proof is based on the following result about invariant integral operators defined by the kernel $\mathcal{F}_{\ell_1,\ell_2} F$.

\begin{proposition} \label{prop:kerneleigen}
 Whenever $\varphi^{(\ell_1)}_V \varphi_V^{(\ell_2)} \not \equiv 0$  we have
 \begin{align*}
     \int_{\G} (\mathcal{F}_{\ell_1,\ell_2} F) (\tau^{-1} \g) \overline{\varphi_{V}^{(\ell_2)}(\g)} d \g  =  \langle F, \mathcal{P}^{(\ell_1,\ell_2)}_{V}\rangle_\G \, \overline{\varphi_{V}^{(\ell_1)}(\tau)}.
 \end{align*} 
Similarly, we have
  \begin{align*}
     \int_{\G} (\mathcal{F}_{\ell_1,\ell_2} F) (\tau^{-1} \g) \overline{E^{(\ell_2)}_{\mathfrak{c}}(\g,\nu) } d \g  =  \langle F, \mathcal{P}^{(\ell_1,\ell_2)}_{\nu}\rangle_\G \, \overline{E^{(\ell_1)}_{\mathfrak{c}}(\tau,\nu) }.
 \end{align*} 
\end{proposition}

\begin{proof}
The second case with the Eisenstein series is similar so we only consider the first. For $\ell_1=\ell_2$ the result follows from  \cite[Proposition 5.2]{GM}, since for $\ell_1=\ell_2$ we have $\PP_V^{(\ell_1,\ell_2)}(\g) = \phi_{\ell_1,\ell_2}(\g,\nu_V)$. For $\ell_1\neq \ell_2$ we will reduce the proof to the case $\ell_1=\ell_2$. By \eqref{eq:Gtwoells} and by integration by parts  \eqref{eq:integrationbypartsraisinglowering}, we get
\begin{align*}
    \int_{\G} (\mathcal{F}_{\ell_1,\ell_2} F) (\tau^{-1} \g) \overline{\varphi_{V}^{(\ell_2)}(\g)} d \g &=  \overline{G(V,\ell_1,\ell_2)}    \int_{\G} (\mathcal{F}_{\ell_1,\ell_2} F) (\tau^{-1} \g) \overline{\mathrm{e}^{(\ell_2-\ell_1)/2}   \varphi_{V}^{(\ell_1)}(\g)} d \g  \\
   & = (-1)^{(\ell_2-\ell_1)/2} \overline{G(V,\ell_1,\ell_2) }    \int_{\G} \mathrm{e}^{(\ell_1-\ell_2)/2}(\mathcal{F}_{\ell_1,\ell_2} F) (\tau^{-1} \g) \overline{   \varphi_{V}^{(\ell_1)}(\g)} d \g
\end{align*}
The function $\mathrm{e}^{(\ell_1-\ell_2)/2}(\mathcal{F}_{\ell_1,\ell_2} F)  $ is of the same left and right type $\ell_1$. Therefore, applying the case $(\ell_1,\ell_1)$ we get
\begin{align*}
     \int_{\G} (\mathcal{F}_{\ell_1,\ell_2} F) (\tau^{-1} \g) \overline{\varphi_{V}^{(\ell_2)}(\g)} d \g  =  (-1)^{(\ell_2-\ell_1)/2} \overline{G(V,\ell_1,\ell_2) }   \langle \mathrm{e}^{(\ell_1-\ell_2)/2}(\mathcal{F}_{\ell_1,\ell_2} F) , \phi_{\ell_1,\ell_1}(\cdot,\nu_V)\rangle_\G \overline{\varphi_{V}^{(\ell_1)}(\tau)}
\end{align*}
Another application of integration by parts \eqref{eq:integrationbypartsraisinglowering} gives us now
\begin{align*}
     \int_{\G} (\mathcal{F}_{\ell_1,\ell_2} F) (\tau^{-1} \g) \overline{\varphi_{V}^{(\ell_2)}(\g)} d \g  =&  \overline{G(V,\ell_1,\ell_2) }   \langle (\mathcal{F}_{\ell_1,\ell_2} F) , \mathrm{e}^{(\ell_2-\ell_1)/2}\phi_{\ell_1,\ell_1}(\cdot,\nu_V)\rangle_\G \overline{\varphi_{V}^{(\ell_1)}(\tau)} \\
     =&  \overline{\Bigl(\frac{G(V,\ell_1,\ell_2) }{G_\phi(\nu_V,\ell_1,\ell_2)} \Bigr)}  \langle (\mathcal{F}_{\ell_1,\ell_2} F) , \phi_{\ell_1,\ell_2}(\cdot,\nu_V)\rangle_\G \overline{\varphi_{V}^{(\ell_1)}(\tau)}\\
     =&\langle F, \mathcal{P}^{(\ell_1,\ell_2)}_{V}\rangle_\G \overline{\varphi_{V}^{(\ell_1)}(\tau)},
\end{align*}
where the last equality holds by \eqref{eq:ratioPphi}. Here we were able to drop the projection operator $\mathcal{F}_{\ell_1,\ell_2}$ in the last step, since the projection is automatic by the Cartan decomposition and the fact that $\mathcal{P}^{(\ell_1,\ell_2)}$ is of left type $\ell_1$ and right type $\ell_2$.
\end{proof}

\subsection{Proof of Proposition \ref{prop:kernelexpansion}}
Proposition \ref{prop:kernelexpansion} now follows from Propositions \ref{prop:specdecompL2}  and \ref{prop:kerneleigen}  by exactly the same argument as \cite[Proposition 5.1]{GM} follows from \cite[Proposition 5.2]{GM}.
\qed

\section{Estimates for harmonics on $\G$}
In this section we provide some basic estimates for the functions $\PP_\nu^{(\ell_1,\ell_2)}$. We begin with the following explicit representation.
\begin{lemma}\label{lem:phirewritten}
For $\ell_1 \equiv \ell_2 \pmod{2}$ and $ \cosh(\varrho) = 2u+1$  we have
    \begin{align*}
       \phi_{\ell_1,\ell_2}(\a_u,\nu) &= \frac{1}{ 2\pi } \int_0^{2\pi } (2u+1 + 2\sqrt{u(u+1)} \cos \varphi)^{-1/2-\nu}  
 \bigg(\frac{\sqrt{1+1/u} + e^{-i\varphi}}{\sqrt{1+1/u} + e^{i\varphi}}\bigg)^{\ell_2/2} e^{i\frac{\ell_2-\ell_1}{2} \varphi } \d \varphi\\
 &= \frac{1}{2 \pi } \int_0^{2\pi} \Bigl(\cosh \frac{\varrho}{2}+\sinh \frac{\varrho}{2} e^{i \varphi}\Bigr)^{-1/2-\nu-\ell_2/2}\Bigl(\cosh \frac{\varrho}{2}+\sinh \frac{\varrho}{2} e^{-i\varphi}\Bigr)^{-1/2-\nu+\ell_2/2} e^{i\frac{\ell_2-\ell_1}{2} \varphi } \d \varphi.
    \end{align*}
    \end{lemma}
\begin{proof}

By \eqref{eq:2pitopi} we have
\begin{align*}
    \phi_{\ell_1,\ell_2}(\a_u,\nu) =  \frac{1}{ \pi} \int_0^{\pi}   \phi_{\ell_2}(\k[\varphi]\a_u,\nu) e^{-i\ell_1 \varphi} \d \varphi. 
\end{align*}
Writing $\k[\varphi]\a[e^{-\varrho}]=\n[x]\a[y]\k[\theta]$ we have by Lemma \ref{le:thetaphiSU}
\begin{align*}
    e^{i\theta}=e^{ i\varphi}\bigg(\frac{\sqrt{1+1/u} + e^{-i2\varphi}}{\sqrt{1+1/u} + e^{i2\varphi}}\bigg)^{1/2}
\end{align*}
and by \eqref{eq:matrixtoiwasawa} we have
\begin{align*}
     y&= \frac{1}{c^2+d^2}= \frac{1}{\cosh(\varrho) + \sinh (\varrho) \cos 2 \varphi}=\frac{1}{2u+1 + 2\sqrt{u(u+1)} \cos 2 \varphi}.
\end{align*}
By the definition of $\phi_\ell(\g,\nu)$ in Iwasawa coordinates \eqref{eq:phibasicdefinition} and the change of variables $\varphi \mapsto \varphi/2$, the first equality follows. The second equality follows by $\cosh(\varrho) = 2u+1$. 
\end{proof}

We have the following simple  uniform upper bound.
\begin{lemma} \label{le:phitrivialbound}
    For $\Re \nu  \in [0,1/2) $ we have
\begin{align*}
      \phi_{\ell_1,\ell_2}(\a_u,\nu)  \ll \min \{  \log (1+u)  , (\Re\nu)^{-1}\} (1+u)^{-1/2 + \Re \nu}.
\end{align*}
\end{lemma}
\begin{proof}
    By Lemma \ref{lem:phirewritten} and the triangle inequality
\begin{align*}
 | \phi_{\ell_1,\ell_2}(\a_u,\nu)| &\leq  \int_0^{2\pi } (2u+1 + 2\sqrt{u(u+1)} \cos \varphi)^{-1/2-\Re\nu}  
\d \varphi \\
&= (2u+1)^{-1/2-\Re\nu} \int_0^{2\pi } (1 + \sqrt{1-1/(2u+1)^2} \cos \varphi)^{-1/2-\Re\nu} 
\d \varphi. 
\end{align*}
By using the lower bound $  \cos(\varphi) \geq \min\{-1+\frac{1}{4}(\varphi-\pi)^2,0\}$,
we have
\begin{align*}
    \int_0^{2\pi } (1 + \sqrt{1-1/(2u+1)^2}& \cos \varphi)^{-1/2-\Re\nu} 
\d \varphi \\
\leq &\int_0^{2\pi } (1 + \sqrt{1-1/(2u+1)^2}  \cdot 0)^{-1/2-\Re\nu}   
\d \varphi \\
&+ \int_0^{2\pi } (1 + \sqrt{1-1/(2u+1)^2} (-1+\tfrac{1}{4}(\varphi-\pi)^2))^{-1/2-\Re\nu}   
\d \varphi.
\end{align*}
The first integral is $\ll 1$, which is sufficiently small. By using the upper bound $  \sqrt{1-1/(2u+1)^2} \leq   1-\frac{1}{4(2u+1)^2},$ the second integral is
\begin{align*}
    &\ll  \int_0^{\pi } (\tfrac{1}{4(2u+1)^2} +\tfrac{1}{4}(\varphi-\pi)^2)^{-1/2-\Re\nu}   \d \varphi \ll \min \{  \log (1+u)  , (\Re\nu)^{-1}\} (1+u)^{2 \Re \nu},
\end{align*}
matching the proposed bound.
\end{proof}

For the exceptional spectrum and the case of $\ell_1=\ell_2$ we need the following more precise version, which also gives a lower bound. The precise dependency  on $\ell$ is not critical for us as long as it is polynomial.

\begin{lemma}\label{lem:simplephiasymp}
Let $\nu \in (1/(\eta \log u),1/2-\eta]$ for some small $\eta>0$. Then for $u > 1/\eta$ and any $\ell \in \Z$
\begin{align*}
      \phi_{\ell,\ell}(\a_u,\nu)  \asymp  \begin{dcases}
          \frac{ u^{-1/2+\nu}}{\nu(1+|\ell|^{2\nu})} \Bigl(1+O\Bigl(\frac{1+\ell^{6}}{u^{1/2}} +   u^{-\nu/10}\Bigr)\Bigr), \quad \ell \text{ even} \\
            \frac{ u^{-1/2+\nu}}{(1+|\ell|^{2\nu})} \Bigl(1+O\Bigl(\frac{1+\ell^{6}}{u^{1/2}} +   \nu^{-1} u^{-\nu/10}\Bigr)\Bigr), \quad \ell \text{ odd}.
      \end{dcases}  
\end{align*}
\end{lemma}
\begin{proof}
Denoting  $R=u^{1/10}(1+\ell)$, we may assume that $R/u < 1$ as otherwise the claim is trivial. We apply Lemma \ref{lem:phirewritten} and split the integral into two parts
    \begin{align*}
       & \int_0^{2\pi } (2u+1 + 2\sqrt{u(u+1)} \cos \varphi)^{-1/2-\nu}    \bigg(\frac{\sqrt{1+1/u} + e^{-i\varphi}}{\sqrt{1+1/u} + e^{i\varphi}}\bigg)^{\ell/2} 
 \d \varphi  \\
 &=    (2u+1)^{-1/2-\nu} \bigg(\int_{|\varphi -\pi | > R/u} + \int_{|\varphi-\pi| \leq R/u}\bigg) .
    \end{align*}
The contribution from the first integral is, by the triangle inequality and by a similar argument as in the proof of Lemma \ref{le:phitrivialbound},
\begin{align*}
    &\leq  (2u+1)^{-1/2-\nu}\int_{|\varphi -\pi | > R/u}  (2u+1 + 2\sqrt{u(u+1)} \cos \varphi)^{-1/2-\nu} \d \varphi 
 \\
 &\ll (2u+1)^{-1/2-\nu} \bigg( 1 +  \int_{|\varphi -\pi| > R/u} (\tfrac{1}{4(2u+1)^2} +\tfrac{1}{4}(\varphi-\pi/2)^2)^{-1/2-\nu}   
\d \varphi  \bigg)  \\
& \ll_\eta  \frac{u^{-1/2+ \nu}}{ \nu R^{2\nu}}  \ll \frac{u^{-1/2+ \nu}}{ \nu (1+|\ell|^{2\nu})}  u^{-\nu/10} ,
\end{align*}
by our assumed lower bound $\nu$. 

For the second integral we have by Taylor approximation for $|\varphi-\pi| \leq R/u$  and $\nu \leq 1/2$
\begin{align*}
  (1+\sqrt{1-\tfrac{1}{(2u+1)^2}} \cos \varphi)^{-1/2-\nu}
 & = \Bigl(\frac{1}{8 u^2} + \frac{(\varphi-\pi)^2}{2}\Bigr)^{-1/2-\nu} + O( R^4 ) 
\end{align*}
and
\begin{align*}
    \Bigl( \frac{\sqrt{1+1/u} + e^{-i\varphi}}{\sqrt{1+1/u} + e^{i\varphi}} \Bigr)^{\ell/2}=\Bigl(\frac{\frac{1}{2u}-i(\varphi-\pi)}{\frac{1}{2u}+i(\varphi-\pi)}\Bigr)^{\ell/2}+O\Bigl(\frac{R^2 |\ell| }{u}\Bigr).
\end{align*}
The contribution from the error terms is bounded by
\begin{align*}
\ll  u^{-1/2-\nu} \int_{|\varphi-\pi| \leq R/u} R^4 u^{2\nu} \d \varphi \ll   u^{-1/2+\nu}  \frac{R^5}{u} \ll \frac{u^{-1/2+\nu}}{1+|\ell|^{2\nu}} \frac{1+\ell^{6}}{u^{1/2}}.
\end{align*}
Then by the substitution $x=2u(\varphi-\pi)$, it remains to show that 
\begin{align*}
   \int_{-2R}^{2R } ( 1+ x^2) ^{-1/2-\nu}\bigg( \frac{1-ix}{1+ix}\bigg)^{\ell/2} \d x \asymp \frac{1}{1+|\ell|^{2\nu}}(1+ O(\nu^{-1} u^{-\nu/10}))
\end{align*}
We write
\begin{align*}
       \int_{-R}^{R} ( 1+ x^2) ^{-1/2-\nu}\bigg( \frac{1-ix}{1+ix}\bigg)^{\ell/2} \d x  =    &\int_{-\infty }^{\infty} ( 1+ x^2) ^{-1/2-\nu}\bigg( \frac{1-ix}{1+ix}\bigg)^{\ell/2} \d x  \\
       &- \int_{|x| > R} ( 1+ x^2) ^{-1/2-\nu}\bigg( \frac{1-ix}{1+ix}\bigg)^{\ell/2} \d x.
\end{align*}
For the integral over $|x| > R$ we have by triangle inequality
\begin{align*}
     \int_{|x| > R} ( 1+ x^2) ^{-1/2-\nu}\bigg( \frac{1-ix}{1+ix}\bigg)^{\ell/2} \d x  \ll  \frac{1}{\nu R^{2\nu}} \ll  \frac{1}{1+|\ell|^{2\nu}}\nu^{-1} u^{-\nu/10}
\end{align*}
The completed integral may be evaluated as
\begin{align*}
    \int_{-\infty }^{\infty} ( 1+ x^2) ^{-1/2-\nu}\bigg( \frac{1-ix}{1+ix}\bigg)^{\ell/2} \d x  = \begin{dcases}
       2^{1-2\nu} (-1)^n  \cos(\pi \nu) \Gamma(2\nu)\frac{\Gamma(1/2+n-\nu)}{\Gamma(1/2+n+\nu)} , \, &\ell=2n, \\
        2^{1-2\nu} (-1)^{n}\sin(\pi\nu) \Gamma(2\nu)\frac{\Gamma(-n-\nu)}{\Gamma(-n+\nu)}, \, &\ell=2n+1,
    \end{dcases}
\end{align*}
which can be seen by the substitution $x= \tan \alpha$ and an integral representation via trigonometric functions for the reciprocal of the Beta function. 
\end{proof}

For the discrete series we can calculate the $L^2$-norm precisely  for  $\mathcal{P}_V^{(\ell_1,\ell_2)}$.  
\begin{lemma} \label{le:orthogonpol}
 For $\nu=(k-1)/2$ with $k \geq 2$ and $|\ell_1|,|\ell_2| \geq k$, $\ell_1 \equiv \ell_2 \equiv k \pmod{2}$ we have
  \begin{align*}
    \int_0^\infty  |\mathcal{P}^{(\ell_1,\ell_2)}_{\nu}(\a_u)|^2 \d u = \frac{1}{k-1}. 
  \end{align*}
\end{lemma}
\begin{proof}
    By Lemma \ref{lem:phirewritten}, we have that
    \begin{align*}
        \phi_{\ell_1,\ell_2}(\a_u,\tfrac{k-1}{2})=\mathfrak{P}^{-k/2}_{-\ell_1/2,-\ell_2/2}(2u+1),
    \end{align*}
    where $\mathfrak{P}$ is as in \cite[sec. 6.5.4 eq. (1)]{liegroupbook}. Combining this with \cite[sec. 6.5.6 eqs. (8), (8')]{liegroupbook} 
    \begin{align*}
        \PP_{\nu}^{(\ell_1,\ell_2)}(\a_u) = \PP^{-k/2}_{-\ell_1/2,-\ell_2/2}(2u+1),
    \end{align*}
where the $\PP$ on right-hand side corresponds to their notation. The lemma now follows from \cite[sec. 7.8.4 eq. (12)]{liegroupbook}.
\end{proof}

\section{Decay in spectrum}
We need the following proposition, which shows that the transforms $ \langle F, \mathcal{P}^{(\ell_1,\ell_2)}_{\nu}\rangle_{\G}$ decay quickly in $\nu_V,\ell_1,\ell_2$ for smooth $F$ having good derivative bounds.
\begin{proposition}\label{prop:decayspec}
   Let $\delta > 0$ and $U  \geq 1$ and let $F:\G\to \C$ be a smooth compactly supported function, supported on $u(\g) \leq U$ which satisfies for all $0 \leq J_0 +2J_1+2J_2 \leq 2J$ \begin{align}\label{eq:derivbound}
   \Omega^{J_0} \partial_{\varphi}^{J_1} \partial_\vartheta^{J_2} F \ll \delta^{-2J_0-J_1-J_2}.
    \end{align}
 Then for $\nu \in i\R$ and $\ell_1\equiv \ell_2 \pmod{2}$ we have
    \begin{align*}
         \langle F, \mathcal{P}^{(\ell_1,\ell_2)}_{\nu}\rangle_{\G}\ll_J U^{1/2}\frac{ \log(1+U)}{(1+(\delta |\nu|)^2+(\delta \ell_1)^2+(\delta \ell_2)^2)^J},
    \end{align*}
    for $\nu \in (0,1/2)$ we have
     \begin{align*}
         \langle F, \mathcal{P}^{(\ell_1,\ell_2)}_{\nu}\rangle_{\G}\ll_J U^{1/2+\nu} \frac{ \min\{\nu^{-1}, \log(1+U)\}(1+\ell_1)(1+\ell_2)}{(1+(\delta \ell_2)^2+(\delta \ell_1)^2)^J},
    \end{align*}
    and for $\nu=(k-1)/2$, $k \in \Z_{>0}$, $|\ell_1|,|\ell_2| \geq k, \ell_1\equiv \ell_2 \equiv k \pmod{2}$, we have
        \begin{align*}
         \langle F, \mathcal{P}^{(\ell_1,\ell_2)}_{\nu}\rangle_{\G}\ll_J U^{1/2} \frac{1}{(1+(\delta |\nu|)^2+(\delta \ell_1)^2+(\delta \ell_2)^2)^J}.
    \end{align*}
\end{proposition}

\begin{proof}
  Define
   \begin{align*}
       \Xi := 1+ \delta^2 \Bigl((1/4-\Omega) -\partial_{\varphi}^2 - \partial_\vartheta^2\Bigr).
   \end{align*}
    so that by \eqref{eq:casimircartan}, \eqref{eq:derivbound}, and the assumption $U> 1$ we have  $  \Xi^{J} F   \ll_J 1.$ Since $\mathcal{P}^{(\ell_1,\ell_2)}_{\nu}$ is an eigenfunction of $\Omega$ with fixed types, we have
    \begin{align*}
        \Xi^{J} \mathcal{P}^{(\ell_1,\ell_2)}_{\nu} =(1+(\delta \nu)^2+(\delta\ell_2)^2+(\delta \ell_1)^2)^J \mathcal{P}^{(\ell_1,\ell_2)}_{\nu}
    \end{align*}
    Therefore, we have by integration by parts and by \eqref{eq:integrationonG}
     \begin{align*}
        \langle F, \mathcal{P}^{(\ell_1,\ell_2)}_{\nu}\rangle_{\G}&=  \frac{1}{(1+(\delta\nu)^2+(\delta \ell_1)^2+(\delta \ell_2)^2)^J}  \langle F, \Xi^J\mathcal{P}^{(\ell_1,\ell_2)}_{\nu}\rangle_{\G}\\
    &  =  \frac{1}{(1+(\delta \nu)^2+(\delta \ell_1)^2+(\delta \ell_2)^2)^J }\langle \Xi^J F, \mathcal{P}^{(\ell_1,\ell_2)}_{\nu}\rangle_{\G}  \\
 & \ll_J\frac{1}{(1+(\delta \nu)^2+(\delta \ell_1)^2+(\delta \ell_2)^2)^J} \int_{0}^{ U} |\mathcal{P}^{(\ell_1,\ell_2)}_{\nu}(\a_u)|\d u.
   \end{align*}
For $\nu \in i \R$ we obtain by Lemma \ref{le:phitrivialbound}, \eqref{eq:unitarynorm}, and \eqref{eq:Gfactorregular} that
\begin{align*}
\int_{0}^{U} |\mathcal{P}^{(\ell_1,\ell_2)}_{\nu}(\a_u)|\d u  & \ll U^{1/2}\log(1+U).
\end{align*}
For $\nu \in [0,1/2)$  we use again Lemma  \ref{le:phitrivialbound} and \eqref{eq:unitarynorm}, now together with \eqref{eq:Gfactorexcep}, to get
\begin{align*}
    \int_{0}^{U} |\mathcal{P}^{(\ell_1,\ell_2)}_{\nu}(\a_u)|\d u & \ll  (1+\ell_1)(1+\ell_2)  \int_{0}^{U} |\phi_{\ell_1,\ell_2}(\a_u,\nu)|\d u \\
& \ll (1+\ell_1)(1+\ell_2) U^{1/2+\nu}\min\{\nu^{-1}, \log(1+U)\}
\end{align*}
Finally, for $\nu=(k-1)/2$ we use Cauchy-Schwarz and Lemma \ref{le:orthogonpol} to get 
\begin{align*}
    \int_{0}^{U} |\mathcal{P}^{(\ell_1,\ell_2)}_{\nu}(\a_u)|\d u  \leq  U^{1/2} \bigg(\int_{0}^{U} |\mathcal{P}^{(\ell_1,\ell_2)}_{\nu}(\a_u)|^2\d u \bigg)^{1/2}  \ll U^{1/2}.
\end{align*}
\end{proof}

\section{Construction of a spectral majorant}

\subsection{Regular Spectrum}
To handle the regular spectrum, we import \cite[Proposition 6.1]{GM} with the minor modification that the discrete sum over $\alpha$ is replaced by the linear functional $\alpha$. Recall that 
\begin{align*}
    u(\g) = u(\g i,i) = \frac{1}{2}(a^2+b^2+c^2+d^2-2)
\end{align*}
\begin{proposition}\label{prop:regularboundbykernel}
Let  $ T \geq L \geq 1$.   There exists a smooth compactly supported function $k^\mathrm{reg}_{1/T}:\G \to \R $ satisfying $k^\mathrm{reg}_{1/T}(\mathrm{k}^{-1} \g \mathrm{k}) = k^\mathrm{reg}_{1/T}(\g)$ and
 \begin{align*}
    k^\mathrm{reg}_{1/T}(\g)\ll\frac{\mathbf{1}\{ u(\g)\leq 1/T\}}{\sqrt{1+u(\g)}}
 \end{align*}
such that for $\nu\in \R \cup i\R$ and $\ell\equiv \kappa\pmod{2}$ we have
 \begin{align*}
T^2 L \,  \langle   k^\mathrm{reg}_{1/T}(\g),\PP_\nu^{(\ell,\ell)} \rangle_\G \geq \mathbf{1}_{|\nu|\leq T} \mathbf{1}_{|\ell|\leq L}.
 \end{align*}
In particular,  for any $\alpha \in \mathcal{L}_c(\G)$ we have
    \begin{align*}
   \sum_{\substack{ \ell \equiv \kappa \pmod{2} \\ |\ell| \leq L} }    \bigg(\sum_{\substack{V \in \B(\Gamma,\chi)  \\ |\nu| \leq T }} | \langle\varphi_V^{(\ell)}\rangle_{\alpha} |^2& + \sum_{\substack{\mathfrak{c} \in \mathfrak{C}(\Gamma,\chi)}} \frac{1}{4\pi} \int_{-i T}^{iT}  |  \langle E^{(\ell)}_{\mathfrak{c}}(\cdot,\nu)\rangle_\alpha |^2 |\d \nu|  \bigg) \leq  T^2 L \, \langle \alpha|\Delta k^\mathrm{reg}_{1/T} | \alpha \rangle .
    \end{align*}   
\end{proposition}

\subsection{Exceptional Spectrum}
In this section we prove the following proposition,  which handles the case of the exceptional spectrum $\nu \in (0,1/2)$ with a factor of $Z^{\nu_V}$ on the spectral side.
\begin{proposition} \label{prop:exceptionalboundbykernel}
For any $\eta > 0$ there exists a large constant $c_0= c_0(\eta)>0$ such that the following holds. Let $Z \geq 2 L \geq 1.$ There exists a smooth compactly supported function $k^{\mathrm{exc}}_Z:\G \to \R$ satisfying $k^{\mathrm{exc}}_Z(\mathrm{k}^{-1} \g \mathrm{k}) = k^{\mathrm{exc}}_Z(\g)$ and \begin{align}\label{eq:exkernk_zupper}
k^{\mathrm{exc}}_Z(\smqty(a& b \\ c & d)) \ll \mathbf{1}\{  u(g) \leq Z \}
 \end{align}
such that for $\nu \in (\eta,1/2-\eta)$, $|\ell| \leq L$ and $Z^{\nu} > c_0 L^6$ we have
 \begin{align*}
 L^4 \,  \langle   k^{\mathrm{exc}}_Z (\g),\PP_\nu^{(\ell,\ell)} \rangle_\G \geq Z^{\nu_V}
 \end{align*}
and $ \langle   k^{\mathrm{exc}}_Z(\g),P_\nu^{(\ell,\ell)} \rangle_\G  \geq 0$ for all $\nu \in \R \cup i \R$ and $\ell \in \Z$. In particular, for any $\alpha \in \mathcal{L}_c(\G)$ we have
\begin{align*}
     \sum_{\substack{ \ell \equiv \kappa \pmod{2} \\ |\ell| \leq L} }    \sum_{\substack{V \in \B(\Gamma,\chi)  \\ \nu_V \in (0,1/2) \\  Z^{\nu_V}\geq c_0 L^6 }} Z^{\nu_V} | \langle\varphi_V^{(\ell)}\rangle_{\alpha} |^2 \leq L^4 \langle \alpha |\Delta k^{\mathrm{exc}}_{Z} 
 | \alpha \rangle.
\end{align*}
\end{proposition}

Let us define
\begin{align*}
     f_1 \rhd f_2 (\g) := \int_\G f_1(\h ) \overline{f_2 (\h \g )} \d \h
\end{align*}
Then from the proof of \cite[Lemma 6.4]{GM} we get the following.
\begin{lemma} \label{le:convosquare}
   Let $k = f \rhd f$ for some $f$ which satisfies $f(\k \g \k^{-1})= f (\g)$. Then
   \begin{align*}
        \langle k, \PP_\nu^{(\ell,\ell)} \rangle_\G =  |\langle f, \PP_\nu^{(\ell,\ell)} \rangle_\G |^2
   \end{align*}
\end{lemma}

We now calculate the behavior of a convolution of a bump function for $u(\g) \asymp \sqrt{Z}$.
\begin{lemma}\label{lem:k_Zupperbound}
 Let $Z > 1$ and $\psi_Z(\g):=\frac{1}{Z^{1/4}}\psi(u(\g)/\sqrt{Z})$ where $\psi$ is a smooth non-negative function supported on $[1/2,4]$ and $1$ on $[1,2]$ and set
 \begin{align*}
     k_Z(\g) := \psi_Z \rhd \psi_Z(\g).
 \end{align*}
 Then
 \begin{align*}
     k_Z(g) \ll \frac{\mathbf{1}\{  u(g) \ll Z \}}{\sqrt{1+u(g)}}
 \end{align*}
\end{lemma}
\begin{proof}
   Since $\psi_Z$ is a bi-$\K$-invariant function, $k_Z$ is also a bi-$\K$-invariant function. Then it suffices to prove that for $\g = \a_u$ we have
\begin{align*}
    \int_{\G} \psi_Z(\h) \psi_Z(\h \a_u ) \d \h \ll   \frac{\mathbf{1}\{ u \ll Z\}}{\sqrt{1+u}}.
\end{align*}
We have by \eqref{eq:integrationonG}
\begin{align*}
    \int_{\G} \psi_Z(\h) \psi_Z(\h \a_u ) \d \h \ll   Z^{-1/4} \int_{\K} \int_{\frac{1}{2}Z^{1/2}}^{4 Z^{1/2}}  \psi_Z(\a_v \k[\theta] \a_u ) \d v \d \k[\theta].
\end{align*}
Recall that $\a_v = \a[e^{-\varrho_v}] =: a[v']$, where $\varrho_v > 0$ with $\cosh(\varrho_v) = 2v+1$. Then by \eqref{eq:ulowerbound} we have $v' \asymp (1+v)^{-1} \asymp Z^{-1/2}$. Similarly, if we denote $a_u = a[u']$ with $u'< 1$ we have $u' \asymp (1+u)^{-1}$.
Then
\begin{align*}
 \a[v'] \k[\theta] \a[u'] = \mqty(\sqrt{u'v'} \cos \theta & \sqrt{\frac{u'}{v'}}\sin \theta \\ -\sqrt{\frac{v'}{u'}}\sin \theta &  \frac{1}{\sqrt{u'v'}}\cos \theta),
\end{align*}
so that in the support of $\psi_Z(\a_v \k[\theta] \a_u ) $ we have
\begin{align*}
    u(\a_v \k[\theta] \a_u) =\frac{1}{4} \Bigl((u'v' + \frac{1}{u'v'}) \cos^2 \theta  + (\frac{u'}{v'}+\frac{v'}{u'}) \sin^2 \theta  -2 \Bigr) \asymp \sqrt{Z}.
\end{align*}
By $v' \asymp Z^{-1/2}$ this implies that
\begin{align*}
  |\sin \theta|^2 \ll \sqrt{Z} v'u' \ll  u' \asymp (1+u)^{-1} \quad \text{and} \quad u' \gg Z^{-1}.
\end{align*}
Therefore,  we get
\begin{align*}
 Z^{-1/4}   \int_{\K} \int_{\frac{1}{2}Z^{1/2}}^{4 Z^{1/2}}  \psi_Z(\a_v \k[\theta] \a_u ) \d v \d \k[\theta] \ll  Z^{-1/2}\mathbf{1}_{u \ll Z} \int_{|\sin \theta|\ll 1/\sqrt{1+u}} \int_{\frac{1}{2}Z^{1/2}}^{4 Z^{1/2}}  \d v \d \theta   \ll  \frac{\mathbf{1}_{u \ll Z}}{\sqrt{1+u }}.
\end{align*}
\end{proof}

\begin{proof}[Proof of Proposition \ref{prop:exceptionalboundbykernel}]
We take in Cartan coordinates
\begin{align*}
    f_Z(\g) :=  \psi_Z(\g) F(\varphi + \vartheta),
\end{align*}
with $\psi_Z$ as in Lemma \ref{lem:k_Zupperbound} and where $F:\R \to [0,1]$ is a smooth $2\pi$-periodic even function with for some large $C>0$
\begin{align*}
       F(x)&=1, \, x \in  [-(2CL)^{-1},(2CL)^{-1}], \\
       \mathrm{supp}(F) &\subseteq [-(CL)^{-1},(CL)^{-1}].
\end{align*}
We set $k^{\mathrm{exc}}_Z = f_Z \rhd f_Z$. By the triangle inequality we have $|k^{\mathrm{exc}}_Z| \leq \psi_Z \rhd \psi_Z$, so that $k_Z$ fulfills the stated upper bound \eqref{eq:exkernk_zupper} by Lemma \ref{lem:k_Zupperbound}. We have
\begin{align*}
    \langle f_Z, \PP_\nu^{(\ell,\ell)} \rangle_\G  &= 2\int \mathcal{F}_{\ell,\ell} f_Z(\a_u) \PP_\nu^{(\ell,\ell)}(\a_u) \d u = 2\hat{F}(\ell) \int \psi_Z(\a_u) \PP_\nu^{(\ell,\ell)}(\a_u) \d u.
\end{align*}
For $C$ sufficiently large we have by Taylor approximation $\hat{F}(\ell) \gg L^{-1}$ for $|\ell| \leq L$ . By Lemma \ref{lem:simplephiasymp} for $Z^\nu \geq c_0 L^6$ and $\nu Z^{\nu/10} \geq  \eta c_0^{1/10} \gg 1$ we have
    \begin{align*}
         \int \psi_Z(\a_u) \PP_\nu^{(\ell,\ell)}(\a_u) \d u \gg \frac{1}{Z^{1/4} (1+|\ell|)} \int_{Z^{1/2}}^{2Z^{1/2}} u^{-1/2+\nu} \d u \ll \frac{Z^{\nu/2}}{L}.
    \end{align*}
Thus, for $\nu \in (\eta,1/2-\eta)$ with  $ Z^\nu \geq c_0 L^6$ we get by Lemma \ref{le:convosquare}
     \begin{align*}
       L^4 \langle k^{\mathrm{exc}}_Z, \PP_\nu^{(\ell,\ell)} \rangle_\G =  L^4 |\langle f_Z, \PP_\nu^{(\ell,\ell)} \rangle_\G |^2 \gg Z^\nu.
   \end{align*}
\end{proof}

\section{Main  theorem}
Recalling the definition of a Hecke operator \eqref{eq:heckeopdef}, for any function $f:\G\times \G\to \C$ and $h_1,h_2 \geq 1$ we define the two variable Hecke operator
\begin{align*}
    (\mathcal{T}_{h_1,h_2} f)(\g_1,\g_2)= (\mathcal{T}_{h_1})_{\g_1}(\mathcal{T}_{h_2})_{\g_2} f(\g_1,\g_2),
\end{align*}
where $(\mathcal{T}_{h})_{\g_i}$ is the Hecke operator $\mathcal{T}_{h}$ acting on the $\g_i$ coordinate. Theorem \ref{thm:mainblackbox} is contained as special case of the following, after restricting to $h_1=h_2=1$.

Analogous to the choice $X_0X_1X_2$, we have a choice with the Hecke operators $\mathcal{T}_{h}$, either treat them as a part of the linear functionals or use an $L^\infty$-bound on the spectral side. The first option is already contained in Theorem \ref{thm:mainblackbox}, the second option is covered by the following. Naturally, in the case that, say, $h_1$ factorises as $h_1=h_0h_1'$ with $\gcd(h_0,h_1')=1$, the following also contains the hybrid version by using the multiplicativity relation $\mathcal{T}_{h_1} = \mathcal{T}_{h_0}\mathcal{T}_{h_1'}$ and absorbing one Hecke operator into the linear functional. We see that $h_i$ and $X_i$ behave analogously, the only difference is that for the $X_i$ we can choose an arbitrary factorisation whereas for $h_i$ we require a factorisation over integers.

\begin{theorem}\label{thm:technical}
Let $\Gamma$ be a congruence subgroup of level $q$ and let $\chi$ be a group character for $\Gamma$. Let $\alpha_1,\alpha_2 \in \mathcal{L}_c(\G)$.  Let $A,C,D > 0$ and $\delta \in (0,1)$  with $AD > \delta$ and let $f \in C^{J}_\delta (A,C,D)$ with $J \geq 10$ even. Denote $F:\G \to \C, \, F(\smqty(a & b \\ c& d)):= f(a,c,d)$ and $R_1=A/C, R_2=D/C$.  

Let $\beta_1(h), \beta_2(h)$ be complex coefficients with finite support.  Denote
\begin{align*}
    \mathcal{H} = \max_{\lambda \in \{\lambda_V ,\lambda_{\c,\nu}\}} |\mathcal{H}(\lambda)|, \quad  \mathcal{H}(\lambda) = \sum_{\substack{h_1,h_2\\ \gcd(h_1h_2,q)=1}}\beta_1(h_1)\beta_2(h_2) \overline{\lambda(h_1)} \lambda(h_2).
\end{align*}
Then for any choice $X_0X_1X_2 \geq AD+1 $ with $X_0,X_1,X_2 \geq 1$ we have 
\begin{align*}
 \sum_{\substack{ \gcd(h_1h_2,q)=1}}&\beta_1(h_1)\beta_2(h_2)  \langle \alpha_1 | \mathcal{T}_{h_1,h_2}  \Delta F | \alpha_2 \rangle \\ &\ll_J \delta^{-O(1)} (AD)^{1/2+o(1)} \mathcal{H} X_0^\theta\sqrt{ \, \langle  \alpha_1| \Delta  k_{X_1^2,R_1}| \alpha_1\rangle  \, \langle  \alpha_2 |\Delta k_{X_2^2,R_2} | \alpha_2 \rangle} 
\end{align*}
where $\langle \alpha_1| \Delta k_{X_i^2,R_i}| \alpha_1\rangle \geq 0$ with 
\begin{align*}
    k_{Y,R}(\g)=k_Y(\a[R]^{-1}\g \a [R]) 
\end{align*}
 for some smooth functions $k_Y:\G \to [0,1]$ which satisfy for any $0 \leq J_0+J_1+J_2 \leq J$
\begin{align} \label{eq:kernelbehaviour}
    \partial_c^{J_0} \partial_{a}^{J_1} \partial_{d}^{J_2}  k_{Y}(\g) \leq \delta^{-O(J_0+J_1+J_2)}|a|^{-J_0}|  c|^{-J_1} | d|^{-J_2}\frac{\mathbf{1}\{ u(\g)\leq Y\}}{\sqrt{1+u(\g)}}.
\end{align}

\end{theorem}

\begin{proof} \hspace{1cm}
\subsection{Unskewing}
We first reduce the proof to the case $R_1=R_2=1$. Denoting
\begin{align*}
    F_1(\g) := F(\a[R_1] \g \a[R_2]^{-1})
\end{align*}
and defining $\alpha_j'$ via (recall \eqref{eq:lrdef})
\begin{align*}
 \langle f \rangle_{\alpha_j'} =   \langle r_{\a[R_j]}f \rangle_{\alpha_j}, 
\end{align*}
we have
\begin{align*}
 \langle \alpha_1|\mathcal{T}_{h_1,h_2}  \Delta F | \alpha_2\rangle =      \langle \alpha_1'|\mathcal{T}_{h_1,h_2}  \Delta F_1 | \alpha_2'\rangle, \quad  \langle \alpha_i|  \Delta  k_{X_i^2,R_i} | \alpha_i\rangle = \langle   \alpha_i'|\Delta  k_{X_i^2,1} | \alpha_i'\rangle.
\end{align*}
Furthermore, $F_1\smqty(a & b \\ c& d))= f_1(a,c,d)$ for some $f_1 \in C^{J}_\delta(\sqrt{AD},\sqrt{AD},\sqrt{AD})$.  Thus, it suffices to prove the theorem for $R_1=R_2=1$, that is $A=D=C$.

\subsection{Spectral expansion}
By the assumption $A=D=C$, we see that $F(\g)$ satisfies \eqref{eq:derivbound} with $\delta_1=\delta^2$. Indeed, the derivatives $\partial_\varphi,\partial_\vartheta$ can be bounded using \eqref{eq:cartantobruhat} since  $r_1,r_2 \asymp 1$, and the differentiation by $\Omega$ can be bounded using the Iwasawa coordinates \eqref{eq:casimiriwasawa} by similar fashion as in \cite[Section 7.3]{GM} to make the change of variables.

By Proposition \ref{prop:kernelexpansion} and applying the Hecke operators $\mathcal{T}_{h_1,h_2}$ via \eqref{eq:heckecusp} and \eqref{eq:heckeeis}, we have
\begin{align*}
  \sum_{\substack{\gcd(h_1h_2,q)=1}}&\beta_1(h_1)\beta_2(h_2)  \langle \alpha_1|\mathcal{T}_{h_1,h_2}  \Delta F | \alpha_2 \rangle  \\
    =& \sum_{\substack{\ell_1,\ell_2 \in \Z \\ \ell_1 \equiv \ell_2 \equiv \kappa \pmod{2}}} \sum_{V\in \mathcal{B}(\Gamma,\chi)} \langle F, \mathcal{P}^{(\ell_1,\ell_2)}_{V}\rangle_{\G}\mathcal{H}(\lambda_V) \overline{\langle\varphi_V^{(\ell_1)}\rangle_{\alpha_1}}\langle\varphi_V^{(\ell_2)}\rangle_{\alpha_2}\\
 &+\sum_{\substack{\ell_1,\ell_2 \in \Z \\ \ell_1 \equiv \ell_2 \equiv \kappa \pmod{2}}} \sum_{\mathfrak{c}\in \mathfrak{C}(\Gamma,\chi)} \frac{1}{4 \pi i}  \int_{(0)} \langle F, \mathcal{P}^{(\ell_1,\ell_2)}_{\nu}\rangle_{\G} \mathcal{H}(\lambda_{\nu,\c})  \overline{\langle E^{(\ell_1)}_{\mathfrak{c},\nu}\rangle_{\alpha_1} } \langle E^{(\ell_2)}_{\mathfrak{c},\nu}\rangle_{\alpha_2}  \d \nu.
\end{align*}
By Proposition \ref{prop:decayspec} we have for $J'= J/2\geq 5$
\begin{align*}
     \langle F, \mathcal{P}^{(\ell_1,\ell_2)}_{\nu}\rangle_{\G} \ll (1+AD)^{1/2+ \widetilde{\nu}} \frac{ \log(1+AD)(1+\ell_1)(1+\ell_2)}{(1+(\delta_1|\nu|)^2+(\delta_1 \ell_2)^2+(\delta_1 \ell_1)^2)^{J'} },
\end{align*}
where $\widetilde{\nu} = \nu$ if $\nu \in (0,1/2)$ and $\widetilde{\nu}=0$ otherwise. Therefore, 
\begin{align*}
      \sum_{\substack{h_1,h_2\\ (h_1h_2,q)=1}}&\beta_1(h_1)\beta_2(h_2)  \langle \alpha_1|\mathcal{T}_{h_1,h_2}  \Delta F | \alpha_2 \rangle \ll (1+AD)^{1/2+o(1)} \mathcal{H} X_0^\theta (\mathcal{C}(X_1X_2) + \mathcal{E})
\end{align*}
where
\begin{align*}
    \mathcal{C}(X_1X_2) =  \sum_{\substack{\ell_1,\ell_2 \in \Z \\ \ell_1 \equiv \ell_2 \equiv \kappa \pmod{2}}} \sum_{V\in \mathcal{B}(\Gamma,\chi)}  \frac{ (1+|\ell_1|)(1+|\ell_2|)}{(1+(\delta_1|\nu_V|)^2+(\delta_1 \ell_2)^2+(\delta_1 \ell_1)^2)^{J'} }\bigl|X_1^{\widetilde{\nu}_V}\langle\varphi_V^{(\ell_1)}\rangle_{\alpha_1}\bigr|\,\bigl|X_2^{\widetilde{\nu}_V}\langle\varphi_V^{(\ell_2)}\rangle_{\alpha_2}\bigr|
\end{align*}
and $\mathcal{E}$ is defined similarly for the continuous spectrum noting that $\widetilde{\nu}=0$ there.
\subsection{Applying Cauchy-Schwarz}
By applying Cauchy-Schwarz we get
\begin{align*}
    \mathcal{C}(X_1X_2) + \mathcal{E} \leq \bigl(  \mathcal{C}_1(X_1) + \mathcal{E}_1
 \bigr)^{1/2} \bigl(  \mathcal{C}_2(X_2) + \mathcal{E}_2
 \bigr)^{1/2}
\end{align*}
where
\begin{align*}
    \mathcal{C}_j(X_j) = \sum_{\substack{\ell_1,\ell_2 \in \Z \\ \ell_1 \equiv \ell_2 \equiv \kappa \pmod{2}}} \sum_{\substack{V\in \mathcal{B}(\Gamma,\chi) }}  \frac{  (1+|\ell_1|)(1+|\ell_2|)}{(1+(\delta_1|\nu_V|)^2+(\delta_1 \ell_2)^2+(\delta_1 \ell_1)^2)^{J'} }\bigl|X_j^{\widetilde{\nu}_V}\langle\varphi_V^{(\ell_j)}\rangle_{\alpha_j}\bigr|^2
\end{align*}
and similarly for the continuous spectrum. We have
\begin{align*}
       \mathcal{C}_j(X_j) &\ll \delta_1^{-O(1)}  \sum_{\substack{ \ell \equiv \kappa \pmod{2}}} \sum_{\substack{V\in \mathcal{B}(\Gamma,\chi)  }} X_j^{2\widetilde{\nu}_V} \frac{ 1+|\ell|}{(1+(\delta_1|\nu|)^2+(\delta_1 \ell)^2)^{J'} }\bigl|\langle\varphi_V^{(\ell)}\rangle_{\alpha_j}\bigr|^2 \\
      & \ll \delta_1^{-O(1)} \sum_{\substack{L=2^i \\ T=2^j }} \frac{L}{(1+(\delta T)^2+(\delta L)^2)^{J'} } \sum_{\substack{  \ell \equiv \kappa \pmod{2} \\ |\ell| \leq L}} \sum_{\substack{V\in \mathcal{B}(\Gamma,\chi) \\ |\nu| \leq T }} X_j^{2\widetilde{\nu}_V} \bigl|\langle\varphi_V^{(\ell)}\rangle_{\alpha_j}\bigr|^2 
\end{align*}
The contribution from $\mathcal{E}_j$ are similar so we only consider $\mathcal{C}_j(X_j)$.

\subsection{Bounding by the automorphic kernel}
We let $\eta > 0$ be a small parameter and let $c_0= c_0(\eta)> 1$ denote the constant in Proposition \ref{prop:exceptionalboundbykernel}. For $2\widetilde{\nu} < \eta$ we use the trivial upper bound $X_j^{2\widetilde{\nu}} \leq X_j^{\eta}$. For $2\widetilde{\nu} \in (\eta,1/2)$ with $X_j^{2\nu_V} \leq c_0 L^6$ we insert the upper bound $X_j^{2\nu_V}  \leq c_0 L^6$. These two ranges can be handled by Proposition \ref{prop:regularboundbykernel}. For the final range of $2\widetilde{\nu} \in (\eta,1/2)$ with $X_j^{2\nu_V} > c_0 L^6$ we use Proposition \ref{prop:exceptionalboundbykernel} instead. Thus, denoting $T'=\max\{T,L\}$, there exist $k_{1/T'}$ and $k^{\mathrm{exc}}_{X_j}$ such that for 
\begin{align*}
    k_{X_j,T,L} = k^{\mathrm{reg}}_{1/T'}+k^{\mathrm{exc}}_{X_j} 
\end{align*}
we have
\begin{align*}
     \mathcal{C}_j(X_j) &\ll  X_j^{\eta}\delta_1^{-O(1)}  \sum_{\substack{L=2^i \\ T=2^j }} \frac{(1+L)(L (T')^2+L^7)}{(1+(\delta_1 T)^2+(\delta_1 L)^2)^{J'} }  \langle \alpha_j |\Delta k_{X_j,T,L} | \alpha_j \rangle.
\end{align*}
We have
\begin{align*}
 \sum_{\substack{L=2^i \\ T=2^j }} \frac{(1+L)(L (T')^2+L^7)}{(1+(\delta_1 T)^2+(\delta_1 L)^2)^{J'} }  k_{X_j,T,L} = \delta_1^{-O(1)} k_{X_j}
\end{align*}
for $k_{X_j}$ which satisfies \eqref{eq:kernelbehaviour}. The claim then follows by taking $\eta = o(1)$. Here the sum converges since $2J' \geq 10$.
The derivative behavior of the kernel function is obvious from their construction via bump functions.

\end{proof}

\section{A variant for right $\mathrm{K}$-invariant functions}
For the applications in \cite{GMquadratic} we will need the following variant for smooth right $\mathrm{K}$-invariant functions, that is, smooth functions of fixed right type $0$. For the application to Theorem \ref{thm:x2y3}, we also require the flexibility to allow the second functional to depend on the variable $h$ in the Hecke operator.

\begin{theorem}\label{thm:kinvtechnical}
    Let $\Gamma$ be a congruence subgroup of level $q$ and let $\chi$ be a group character for $\Gamma$.   Let $\mathbf{X},\mathbf{Y} > 0$ and $\delta \in (0,1)$ and suppose that
  $\mathbf{X}/\mathbf{Y}  > \delta$.   Let $f \in C^{10}_\delta (\mathbf{X},\mathbf{Y})$. Denote $F:\G \to \C, \, F(\g):= f(x,y)$ for $\g = \n[x]\a[y]\k$.  
Let $H \geq 1$ and let $\beta(h)$ be complex coefficients supported on $h \leq H$. Let  $\alpha_1,\alpha_{2,h} \in \mathcal{L}_c(\G)$.

Then for any choice $Z_0Z_1Z_2\geq \mathbf{X}/\mathbf{Y}+1$ with $Z_0,Z_1,Z_2 \geq 1$ we have 
\begin{align*}
 \sum_{\substack{\gcd(h,q)=1}}\beta(h)  \langle \alpha_1 | \mathcal{T}_{h,1}  \Delta F | \alpha_{2,h} \rangle \ll & \, q^{o(1)} \delta^{-O(1)}  (\mathbf{X}/\mathbf{Y})^{1/2+o(1)} H^{1/2} Z_0^\theta  \, \langle  \alpha_1| \Delta  k_{Z_1^2,\mathbf{X}}| \alpha_1\rangle^{1/2} \\
& \times\bigg(  \, \sum_{h \leq H} |\beta(h)|^2\langle  \alpha_{2,h} |\Delta k_{Z_2^2,1} | \alpha_{2,h} \rangle \bigg)^{1/2},
\end{align*}
where $\langle \alpha_1| \Delta k_{Z_i^2,R_i}| \alpha_1\rangle \geq 0$ with 
\begin{align*}
    k_{Y,R}(\g)=k_Y(\a[R]^{-1}\g \a [R]) 
\end{align*}
for certain smooth functions $k_{Y,R} :\G \to [0,1]$ which satisfy
\begin{align*}
    k_{Y,R}(\g)  \leq  \frac{\mathbf{1}\{u_R(\g) \leq Y\}}{\sqrt{1+u_R(\g)}}.
\end{align*}
\end{theorem}

\begin{proof}
    The proof is mostly the same  as the proof of Theorem \ref{thm:technical}. In the unskewing step we can take $R_1=\mathbf{X}$ and $R_2=1$ since the function $F$ is already  of fixed right type $0$. The unskewing replaces $f(x,y)$ by $f_1(x,y)$ with $f_1 \in C^{J}_\delta(1,\mathbf{Y}/\mathbf{X})$, and $F_1(\g)=f_1(x,y)$ is supported on $u(\g) \ll 1/y \asymp \mathbf{X}/\mathbf{Y}$.
    
We apply Cauchy-Schwarz with $h$ on the outside and keep the Hecke eigenvalues $\lambda(h)$ with the first kernel associated to $\alpha_1$, which does not depend on $h$. The claim then follows by similar aguments as before, by applying the Rankin-Selberg bound \eqref{eq:rankinselberg}.
\end{proof}

\subsection*{Acknowledgements}
The authors are grateful to James Maynard for helpful discussions and providing the espresso machine. We are also grateful to Alex Pascadi and Jared Duker Lichtman for helpful discussions. The project has
received funding from the European Research Council (ERC) under the European Union's Horizon 2020 research and innovation programme (grant agreement No 851318).

\bibliography{SL2bib}
\bibliographystyle{abbrv}
\end{document}